\documentclass[11pt]{article}
\usepackage{amsmath,amssymb,amsthm,bm}
\usepackage{mathtools}
\usepackage{appendix}

\usepackage{dsfont}
\usepackage{floatflt}
\usepackage{verbatim}
\usepackage{graphicx,rotating}
\usepackage{xcolor}
\usepackage{caption}
\usepackage{natbib}
\usepackage{tikz}
\usepackage{wrapfig}
\usepackage{faktor}
\usepackage[framemethod=tikz]{mdframed}
\usepackage{bbm}
\usepackage[normalem]{ulem} 
\usepackage{enumitem} 
\usepackage{xfrac}
\usepackage{natbib}

\usepackage{hyperref}
\setitemize[1]{leftmargin=9ex}

\usepackage{xr}

\usepackage[refpage]{nomencl}

\def\@@@nomenclature[#1]#2#3#4{%
 \def\@tempa{#2}\def\@tempb{#3}\def\@tempc{#4}%
 \protected@write\@nomenclaturefile{}%
  {\string\nomenclatureentry{#1\nom@verb\@tempa @[{\nom@verb\@tempa}]%
      \begingroup\nom@verb\@tempb\protect\nomeqref{\theequation}%
        |nomlabelref}{\@tempc}}%
 \endgroup
 \@esphack}
\usepackage{authblk}

\makenomenclature

\newcommand{\subscript}[2]{$#1 _ #2$}

\newmdenv[tikzsetting={draw=black, line width=0.5pt, dash pattern=on 1pt off 1pt, dashed}, linecolor=white, outerlinewidth=1pt]{examplebox}

\newtheorem{thm}{Theorem}
\newtheorem{lemma}{Lemma}
\newtheorem{corr}{Corollary}
\newtheorem{example}{Example}

\theoremstyle{definition}
\newtheorem{defn}{Definition}
\newtheorem{assumption}{Assumption}

\theoremstyle{remark}

\newtheorem*{remark}{Remark}

\begin{document}

\title{Regularity Conditions for Critical Point Convergence}

\author[1]{Thomas J. Maullin-Sapey}
\author[2]{Samuel Davenport}
\affil[1]{School of Mathematics, University of Bristol, Bristol, United Kingdom}
\affil[2]{Division of Biostatistics, University of California, San Diego, CA, USA}
\date{}
\maketitle

\begin{abstract}
We focus on a sequence of functions $\{f_n\}$, defined on a compact manifold with boundary $S$, converging in the $C^k$ metric to a limit $f$. A common assumption implicitly made in the empirical sciences is that when such functions represent random processes derived from data, the topological features of $f_n$ will eventually resemble those of $f$. In this work, we investigate the validity of this claim under various regularity assumptions, with the goal of finding conditions sufficient for the number of local maxima, minima and saddle of such functions to converge. In the $C^1$ setting, we do so by employing lesser-known variants of the Poincar\'{e}-Hopf and mountain pass theorems, and in the $C^2$ setting we pursue an approach inspired by the homotopy-based proof of the Morse Lemma. To aid practical use, we end by reformulating our central theorems in the language of the empirical processes.
\end{abstract}


\section{Introduction}

Practical image analysis settings conventionally assume that a sequence of continuous functions, $f_n$, defined on a compact domain $S$, converge in some appropriately chosen metric to a limiting process, $f$ (for instance, see the discussions of \cite{adler1981geometry}, \cite{Cheng2017} and \cite{davenport2022confidenceregionslocationpeaks}). Although often not acknowledged, many empirical studies  implicitly assume that, as more data is collected and $n$ increases, topological features of $f_n$, such as the number of critical points, will converge to those of $f$. Without appropriate assumptions in place, this is not always the case, and, in fact, there are many settings where assuming such convergences hold would be misleading and inaccurate. For instance, images of the night sky found in the cosmic background explorer dataset will exhibit fractal-like properties, with more stars appearing as resolution increases,  limiting the finite resolution interpretation of topological features \citep{COLEMAN1992,Pino1995}. By contrast in medical imaging fields such as neuroimaging, great importance is assigned to topological features such as the number of `peaks' and `clusters' in an image of activity in the human brain derived from a finite number of subjects, $n$ \cite{Friston1994, Chumbley2010, maullin2023spatial}. However little is provided in the way of theoretical guarantees that these features are representative of the limiting average. 

In this work, we aim to answer a fundamental question; under what regularity conditions can we assume that the number of local maxima, minima and saddles of $f_n$ converge to that of $f$? Specifically, we explore the consequences of assuming $f_n\rightarrow f$ in the $C^k$ metric for various $k\in \mathbb{N}$, with an underlying $C^k$ structure placed on the domain $S$. It should be noted that in many real-world settings, authors have had to derive such conditions for specific applications \citep{Adler2009, Cheng2017, Eldar2024}. However, we believe that in such settings the regularity assumptions made can be greatly reduced, and there is a clear need in the literature for a canonical theory, so that authors may bypass such derivations in future.

Although seemingly simple, the problem we consider requires much machinery to resolve. A first approach which we consider is to apply the techniques of Morse theory to draw a correspondence between critical points of $f_n$ and $f$. Doing so is complicated by two factors. Firstly, standard Morse theoretic arguments guarantee the existence of neighborhoods around critical points in which $f_n$ and $f$ behave nicely, but typically do not quantify the size of such neighborhoods. This is important, as to draw a bijection between the critical points of $f_n$ and $f$, we need to rule out cases where $f$ has a critical point with a Morse neighborhood which contains multiple Morse neighborhoods of critical points of $f_n$. Secondly, such an approach does not tell us about the convergence of degenerate critical points, nor anything about what can be said when $f_n$ converges in the $C^1$ metric.

An alternative technique is to use the tools of homology theory. This is an approach which we also pursue, allowing us to analyze more complex degenerate saddles and settings in which $f_n$ and $f$ are only $C^1$. However, this approach again suffers multiple complications. First, to draw a correspondence between critical points, we require a non-conventional notion of homological index for manifolds with boundary, which we draw from \cite{jubin2009generalized}. Second, the homological index does not distinguish between local maxima, minima and saddles in all dimensions. For instance, in two-dimensions the homological index of a local minimum and a local maximum are both one. In order to state something meaningful about local maxima, minima and saddles in such contexts, we shall instead require a variant of the mountain pass theorem for convex domains (c.f. \cite{Jabri_2003} and Appendix \ref{app:mpt}).

This paper is structured as follows. In Section \ref{sec:notation}, we detail the notation we shall use, with Section \ref{sec:homology_notation} introducing homological concepts and Section \ref{sec:morse_notation} outlining Morse-theoretic definitions. Section \ref{sec:theory} then details our theory with illustrative examples given throughout. Specifically, Section \ref{sec:C0} briefly illustrates why convergence in the $C^0$ metric is not sufficient for our purposes. Following this, in Section \ref{sec:homology}, we use the extended Poincar\'{e}-Hopf theorem of \cite{jubin2009generalized}, and a variant of the mountain pass theorem for convex domains provided by \cite{Jabri_2003} to investigate the behavior of the critical points under the assumption of $C^1$ convergence. Then, in Section \ref{sec:morse}, using the variant of Thom's homotopic proof (\cite{arnold1985singularities}) provided by \cite{ioffe1997}, we investigate whether Morse theory can be used to draw a correspondence between critical points under the assumption of $C^2$ convergence. We end with a short section, Section \ref{sec:prob}, adapting our core results into the framework of empirical random processes outlined by \cite{van1996weak}, along with a short discussion in Section \ref{sec:discussion}. Supplemental lemmas are contained in Appendix \ref{app:lemmas}, while the versions of the Poincar\'{e}-Hopf theorem, mountain pass theorem, and Morse lemma that we employ are provided in Appendices \ref{app:ph}, \ref{app:mpt} and \ref{app:morse}, respectively. As we draw on notions from a range of areas, supplemental material is provided acting as a primer on the concepts we reference from empirical random processes, manifold theory and singular homology. For all figures of parameterized surfaces in this document, corresponding interactive Desmos plots are also linked in the supplement.

\section{Notation}\label{sec:notation}

In this section, we list the notation which will be used throughout the document. As we draw from several mathematical domains, a nomenclature is included at the end of the document to help track the notation employed.

The following notation is universally employed in both the document, appendices and supplement. For an arbitrary set $A$, we denote its complement as $A^c$\nomenclature{$A^c$}{Complement of the set $A$}, closure as $\overline{A}$\nomenclature{$\overline{A}$}{Closure of the set $A$}, topological boundary as $\partial A$\nomenclature{$\partial A$}{Topological boundary of the set $A$} and interior as Int$(A)$\nomenclature{Int$(A)$}{Interior of the set $A$}. For an arbitrary point $s \in U\subseteq \mathbb{R}^D$, we denote the closed $\epsilon$-ball around $s$ as $B_{\epsilon}(s):=\{s' \in U: |s-s'|\leq \epsilon\}$\nomenclature{$B_{\epsilon}(s)$}{$\epsilon$-ball around $s$}, where the ambient space $U$ is context-dependent. For brevity, we also adopt the shorthand $B_\epsilon:=B_\epsilon(0)$\nomenclature{$B_{\epsilon}$}{Shorthand for $B_\epsilon(0)$}. 

We assume that we are working on a $D$-dimensional\nomenclature{$D$}{Dimension of $S$} compact $C^k$-manifold\nomenclature{$k$}{Regularity of $S, f$ and $f_n$} with boundary, denoted $S$\nomenclature{$S$}{Compact $C^k$-manifold with boundary}, where $k$ will typically equal either $1$ or $2$. Points of $S$ are denoted with the lowercase letters $x,y,z,p$ or $s$, or variants thereof (e.g. $s',p^*,x_1,...$ etc)\nomenclature{$x,y,z,p,s$}{Arbitrary points in $S$}. When looking at functions on $S$, we shall assume they are of class $C^k$ for some $k\geq 1$, meaning unless stated otherwise, we assume the first derivative is always well-defined and continuous. The definitions we employ follow those of Chapter 1 of \cite{lee2013smooth} (see Supplementary Material \ref{supp:mflds} for further details). 

Given a $C^1$ function, $f:S\rightarrow \mathbb{R}$\nomenclature{$f$}{Scalar function on $S$, assumed to be either $C^1$ or $C^2$ depending on the context}, a point $s \in S$ is said to be critical if $\nabla f(s)=0$\nomenclature{$\nabla f$}{The gradient of $f$}. For ease, we denote the set of zeros of a vector field $v$ on $S$ as $Z(v)$\nomenclature{$Z(v)$}{The set of zeros of a vector field $v$}, e.g. $Z(\nabla f):=(\nabla f)^{-1}(0)$. When $f$ is $C^2$, we denote the Hessian of $f$ at $s\in S$ as $H_f(s)$\nomenclature{$H_f(s)$}{Hessian of $f$ at $s$}\nomenclature{$H, H_n$}{Shorthand for $H_{f}(0)$ and $H_{f_n}(0)$, respectively}. For brevity, we shall sometimes adopt the shorthand $H:=H_f(0)$ and $H_n:=H_{f_n}(0)$. Throughout this work, we shall be considering convergence in the $C^k$ metric, defined as follows:

\begin{mdframed}
    \begin{defn}[Convergence in the $C^k$-metric]
        A sequence of functions $f_n:S\rightarrow \mathbb{R}$ is $C^k$ convergent to $f$$:S \rightarrow \mathbb{R}$ if both $\{f_n\}$ and $f$ are (at least) $C^k$ and all partial derivatives of $f_n$ of order $\{ 0, 1, \dots, k\}$ converge uniformly to those of $f$. We denote this as $f_n\xrightarrow{C^k} f$\nomenclature{$\xrightarrow{C^k}$}{Convergence in the $C^k$ metric}.
    \end{defn}
    \vspace*{0.25cm}
\end{mdframed}
In particular, we have that if $f_n\xrightarrow{C^2} f$, then $|f_n-f|$, $|\nabla f_n-\nabla f|$ and $||H_{f_n}-H_f||$ converge to zero uniformly on $S$, where $|\cdot|$\nomenclature{$\lvert\cdot\lvert$}{Uniform norm on vector valued functions} and $||\cdot||$\nomenclature{$\lvert\lvert\cdot\lvert\lvert$}{Uniform norm on matrix valued functions} represent a choice of vector norm and the matrix norm it induces, respectively. \\

\subsection{The Homological Index}\label{sec:homology_notation}

In this work, we shall consider two conventions for classifying `types' of critical points. The first uses topological degree theory. A short supplement summarizing the key concepts of singular homology we draw upon may be found in Supplemental Material \ref{supp:singhom}. In this subsection, we introduce notation for the proofs of Section \ref{sec:homology}. Here, we assume that $f$ is of at least $C^1$ regularity. 

We now define our first tool for classifying critical points: the \textit{Homological Index}. As the definition we employ is an extension, derived by \cite{jubin2009generalized} for manifolds with boundary, of the standard definition, we outline it below.
\begin{mdframed}
    \begin{defn}[Homological Index on a Manifold]\label{def:homological_index_mfld} Let $v$ be a continuous vector field on $S$ with isolated zeros, $z \in Z(v)\cap \text{Int}(S)$ and let $\epsilon > 0$ be such that $B_\epsilon(z) \cap Z(v)=\{z\}$. Then we define the \textit{homological index of $v$ at $z$} as:
    \begin{equation}\nonumber
        \text{Ind}^H(v,z):= \text{deg}(\varphi) \quad \text{ where } \quad \varphi:\partial B_\epsilon(z)\rightarrow\mathbb{S}^{D-1}\quad  \text{ is defined by }\quad \varphi(s)=\frac{v(s)}{||v(s)||}.
    \end{equation} 
    \nomenclature{$\text{Ind}^H(v,z)$}{The homological index of the vector field $v$ at $z\in \text{Int}(S)$}where deg represents the topological degree\nomenclature{$\text{deg}$}{The topological degree} (see Supplemental Material Section \ref{supp:singhom}). If $Z(v)$ is finite, the \textit{interior homological index of $v$}\nomenclature{$\text{Ind}^{H}_\circ(v)$}{Interior homological index of the vector field $v$ on $S$} is defined as the sum of the homological indices of $v$ over all zeros $z$ in the interior of $S$. That is:
    \begin{equation}\nonumber
        \text{Ind}^H_\circ(v):= \sum_{z \in Z(v) \cap \text{Int}(S)} \text{Ind}^H(v,z).
    \end{equation}
    \end{defn}
    \vspace{0.1cm}
\end{mdframed}
For further detail on this definition, as well as proof that it is not affected by the choice of sufficiently small $\epsilon$, we recommend \cite{lee2010topological} and \cite{guillemin2010differential}. Definition \ref{def:homological_index_mfld} assumes $v\in \text{Int}(S)$ and thus cannot be used to describe the behavior of the vector field $v$ on boundary of $S$. For this reason, the definition must be extended. \\
\\
To do so, we must first introduce the notion of a ``collar" for a manifold with boundary. Following the work of \cite{jubin2009generalized}, we define this as follows:
\begin{mdframed}
    \begin{defn}[Collar of a Manifold with Boundary] Given a compact $C^1$ manifold with boundary, $S$, a collar of $S$ is a diffeomorphism $\phi:U \rightarrow \partial S \times \mathbb{R}_{\geq 0}$ where $U\subset S$ is an open subset containing $\partial S$. The tangential and transversal components of $\phi$ are the maps $\phi_1:U\rightarrow \partial S$ and $\phi_2:U\rightarrow \mathbb{R}_{\geq 0}$ satisfying $\phi=(\phi_1,\phi_2)$. 
    \end{defn}
    \vspace{0.2cm}
\end{mdframed}

\begin{floatingfigure}[r]{0.7\textwidth}\centering
    \includegraphics[width=0.7\textwidth]{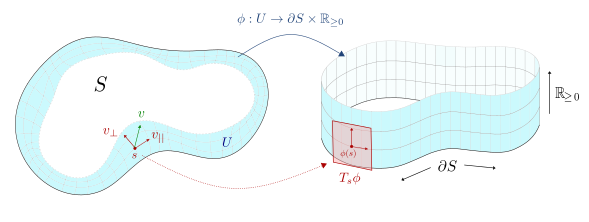}
    \vspace{-0.5cm}\captionof{figure}{The collar of a manifold with boundary. Tangential and transversal directions are indicated by the grey gridlines, with a specific example provided for an arbitrary point $s\in S$ (red). }
    \label{fig:collar}
\end{floatingfigure}

Given a collar $\phi:U\rightarrow \partial S \times \mathbb{R}_{\geq 0}$, any continuous vector field $v$ defined over $U$ can be decomposed into tangential and transversal components with respect to $\phi$, $v_{||}$ and $v_{\perp}$, defined by $v_{||} = T \phi_1 \circ v$ and $v_{\perp} = T \phi_2 \circ v$,
where $Tf$ represents the tangent space of $f$ (see Fig. \ref{fig:collar}). Using these definitions, we may define a notion of \textit{`boundary index of $S$'}.

\begin{mdframed}
    \begin{defn}[Homological Boundary Index] Let $S$ be a compact manifold with boundary and $v$ be a continuous vector field defined on $S$. Let $\phi$ be a collar for $S$ and suppose that, when restricted to $\partial S$, the transversal component of $v$ with respect to $\phi$ has isolated zeros. If $z \in Z(v_{||}|_{\partial S})$ then we define the \textit{homological boundary index of $z$} as follows:
        \begin{equation}\nonumber
            \text{Ind}_\partial^H(v,z)= w(z)\cdot\text{Ind}^H\bigg(v_{||}\big|_{\partial S}, z\bigg) \quad \text{ where }w(z)=\begin{cases}
                \frac{1}{2} & \text{If }v(z)\text{ points into }S\text{,} \\
                -\frac{1}{2} & \text{If }v(z)\text{ points out of }S\text{.} \\
            \end{cases}
        \end{equation}
        If the transversal component of $v$ with respect to $\phi$ does not have isolated zeros then  $v$ and $\phi$ can be perturbed to obtain new vector field $\tilde{v}$ and collar $\tilde{\phi}$ such that the transversal component of $\tilde{v}$ with respect to $\tilde{\phi}$ does have isolated zeros. In this case, we define the index $\text{Ind}_\partial^H(v,z):=\text{Ind}_\partial^H(\tilde{v},z)$.\nomenclature{$\text{Ind}_\partial^H(v,z)$}{The boundary homological index of the vector field $v$ at $z\in \partial S$} \\
        \\
        If $Z(v_{||}|_{\partial S})$ is finite, the \textit{boundary homological index of $v$}\nomenclature{$\text{Ind}^{H}_\partial(v)$}{Boundary homological index of the vector field $v$ on $S$} is defined as:
        \begin{equation}\nonumber
            \text{Ind}^H_\partial(v):= \sum_{z \in Z\big(v_{||}\big|_{\partial S}\big)} \text{Ind}^H_\partial(v,z).
        \end{equation}
    \end{defn}
    \vspace{0.1cm}
\end{mdframed}
That the above definitions are independent of the choice of collar $\phi$ and (suitably small) perturbations is shown by Propositions 8, 9 and 10 of \cite{jubin2009generalized}. The above may now be used to define the notion of homological index on a manifold with boundary.
\begin{mdframed}
    \begin{defn}[Homological Index on a Manifold with Boundary]\label{def:homological_index} Let $v$ be a continuous vector field on a manifold with boundary $S$ with isolated zeros. The \textit{homological index of $v$}\nomenclature{$\text{Ind}^{H}(v)$}{Homological index of vector field $v$ on $S$} is defined as:
    \begin{equation}\nonumber
        \text{Ind}^H(v) := \text{Ind}^H_\partial(v) + \text{Ind}^H_\circ(v).
    \end{equation}
    \end{defn}
    \vspace{0.1cm}
\end{mdframed}
A deep theorem of topology and geometry, originally shown by \cite{H1881} and \cite{Hopf1927}, generalized by \cite{jubin2009generalized} and stated formally in Appendix \ref{app:ph}, is the Poincar\'{e}-Hopf theorem. Broadly speaking this theorem states the index of a vector field on a manifold with boundary $S$ is equal to its Euler characteristic, $\chi(S)$\nomenclature{$\chi$}{Euler characteristic}, if the dimension of $S$ is even and zero otherwise. An illustration of the above construction for two maps defined on the two-dimensional $\epsilon-$ball is provided by Fig. \ref{fig:indexbdry}. As can be seen, for both maps the homological index is $1$, which equals $\chi(B_\epsilon)$.

{\centering
    \includegraphics[width=0.75\textwidth]{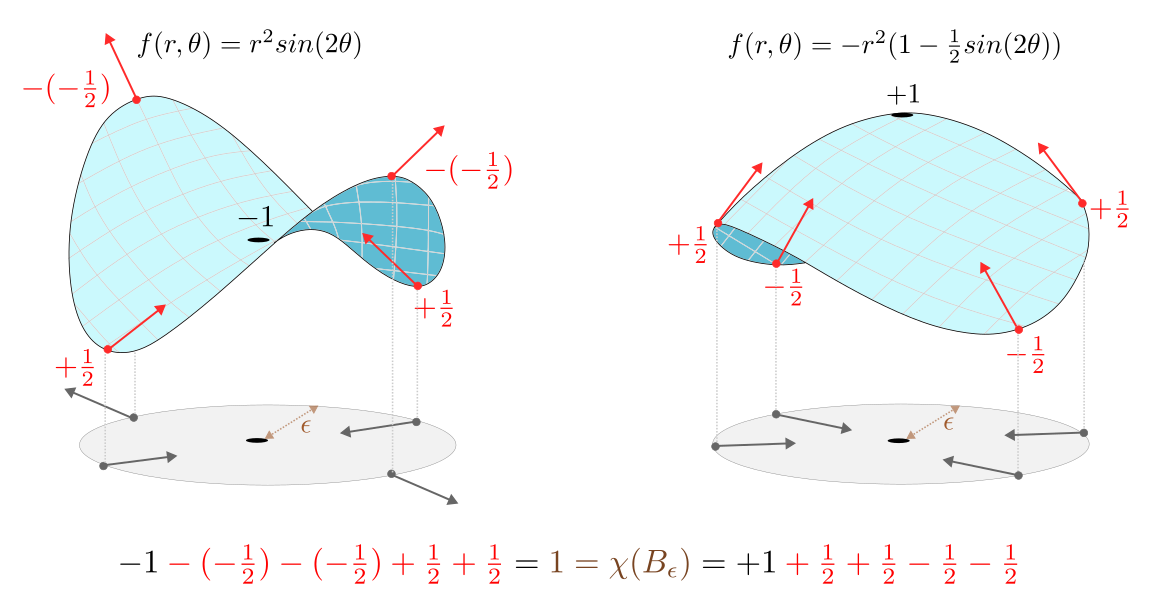}
    \captionof{figure}{Illustration of the computation of $\text{Ind}^{H}(\nabla f)$ for two functions. The contributions from critical points on the boundary are highlighted in red, whilst those of the critical points in the interior are shown in black. As predicted by Poincar\'{e}-Hopf theorem, Theorem \ref{thm:poincarehopf}, in this case $\text{Ind}^{H}(\nabla f)=1=\chi(B_\epsilon)$ for both functions.}
    \label{fig:indexbdry}}
    \vspace{0.2cm}

If $f\in C^1(S, \mathbb{R})$\nomenclature{$C^k(X,Y)$}{Space of $C^k$ functions from $X$ to $Y$}, $s \in \text{Int}(S)$ is a critical point, $\lambda \in \mathbb{Z}$ and $\text{Ind}^H(\nabla f, s)=\lambda$, then we say that $s$ is a critical point of homological index $\lambda$. These definitions will be especially helpful in identifying points of undulation (critical points at which curvature does not change).

Suppose $f\in C^1(S,\mathbb{R})$ has isolated critical points. Formally, a point $p\in S$ is a local maxima of $f$ if it has a neighborhood over which $f<f(p)$, with the analogous definition holding for local minima. If $p\in S$ is a critical point with homological index $0$, we say that it is an undulation point. If $p\in S$ is a critical point that is neither a local maxima, minima, or undulation point, we say that it is a saddle point. As shown in Fig. \ref{fig:homology_index}, for two-dimensional functions, local maxima and minima are critical points of homological index $1$, undulation points have homological index $0$, hyperbolic paraboloid saddle points have homological index $-1$, monkey saddle points have homological index $-2$ and more complex examples of non-degenerate saddles have index $< -2$. \\

{\centering
    \includegraphics[width=0.8\textwidth]{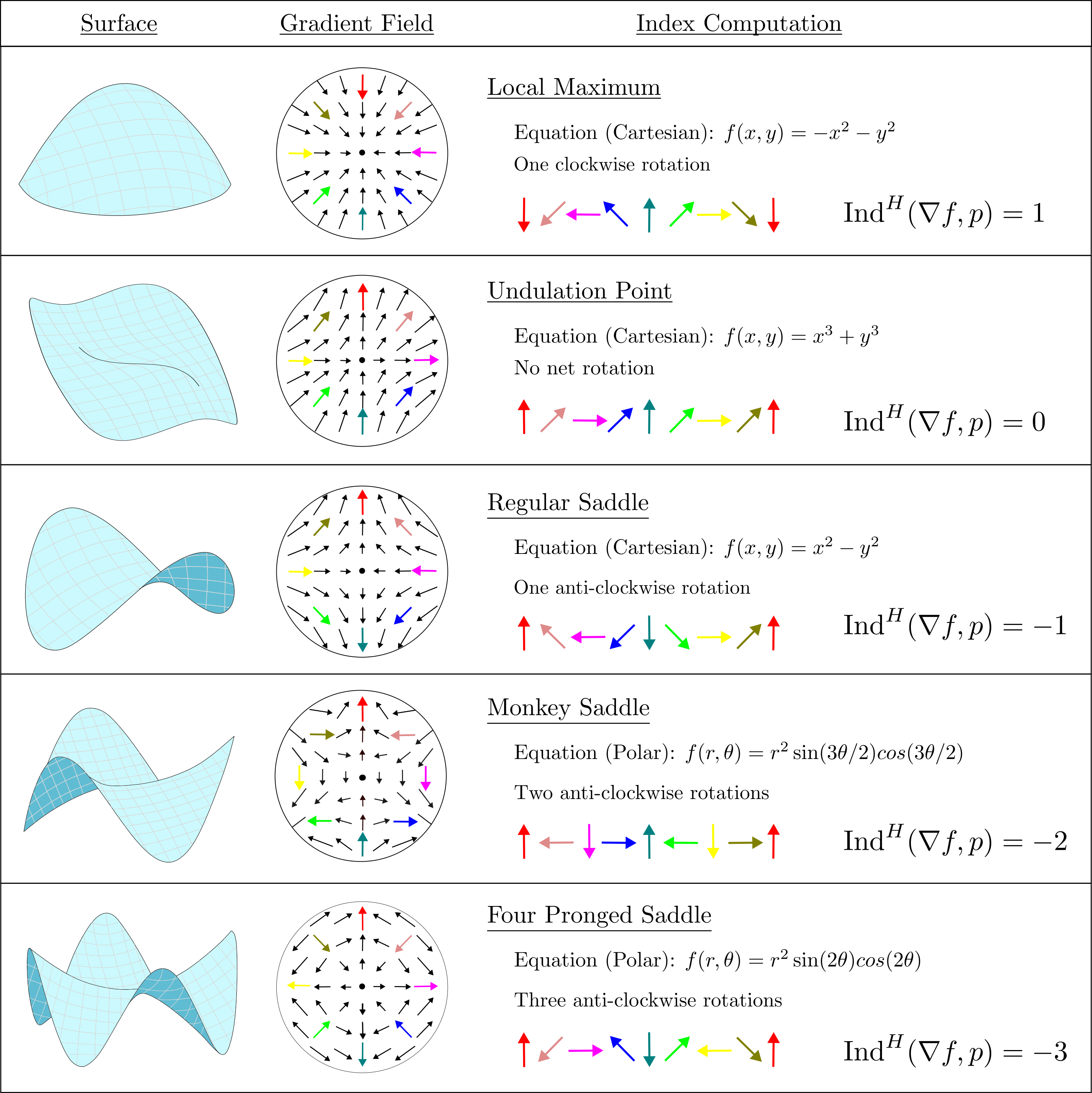}
    \captionof{figure}{Illustration of the homological index for isolated critical points. In each example, a surface plot is shown in blue alongside a top-down view of the gradient field of the surface. In each gradient field, a sequence of vectors, moving around the critical point clockwise are highlighted. The same vectors, standardized, are then laid out horizontally, on the right. To compute the homology index, we count the number of times the vectors rotate clockwise.}
    \label{fig:homology_index}}

\vspace{0.5cm}

By extending the logic illustrated in Fig.~\ref{fig:homology_index}, it can be seen that if $p$ is a saddle point of $f:\mathbb{R}^2\rightarrow\mathbb{R}$, with $j$ `prongs' ($j\geq 2$) then $\text{Ind}^H(\nabla f,p)=-(j-1)$. We note here that it is also possible to generate vector fields with zeros at which the degree equals $2$ or higher. However, in two dimensions, such fields do not define gradients as they necessarily possess non-zero curl, and are thus not of interest to this work (c.f. \cite{guillemin2010differential}). In higher dimensions, the homological index is harder to visualize.

Throughout this work, we shall denote the number of critical points of $f$ that have homological index $\lambda$ as $N^H_\lambda(f)$\nomenclature{$N^H_\lambda(f)$}{Number of critical points of a function $f$ with homological index $\lambda$}. We shall denote $N_C(f)$, $N_M(f)$, $N_m(f)$ and $N_S(f)$\nomenclature{$N_C(f)$}{Number of critical points of a function $f$}\nomenclature{$N_M(f)$}{Number of local maxima of a function $f$}\nomenclature{$N_m(f)$}{Number of local minima of a function $f$}\nomenclature{$N_S(f)$}{Number of saddle points of a function $f$} as the total number of critical points, local maxima, local minima and saddle points of $f$. Our aim is to find minimally sufficient conditions under which these quantities converge for $f_n$ converging to $f$ (e.g. under what conditions does $N_\lambda^H(f_n)\rightarrow N_\lambda^H(f)$?).

\subsection{The Morse Index}\label{sec:morse_notation}

The second notion of classifying critical points which we shall consider is the Morse index. To discuss the Morse index of a critical point $p$ of a function $f$, we assume $f$ is of at least $C^2$ differentiability, with the Hessian matrix of $f$ at point $p\in S$ denoted $H_f(p)$. For a critical point $p$ of the function $f$, if $H_f(p)$ is of full rank, the Morse index, $\text{Ind}^M(\nabla f,p)$\nomenclature{$\text{Ind}^M(\nabla f,z)$}{Morse index of $f$ at $z\in\text{Int}(S)$} is defined as the number of negative eigenvalues of $H_f(p)$ (see Fig. \ref{fig:morse_index}).

When analyzing critical points, one benefit of employing Morse theory over the homological approach is that it can provide richer information about the type of critical point being considered. For instance, in three dimensions, if $\text{Ind}^H(\nabla f,p)=1$ then $p$ could be a saddle or local minima. By contrast, the Morse index tells us exactly what type of critical point $p$ is, allowing us to distinguish between saddles, minima and maxima regardless of dimension. 

On the other hand, Morse theory can also be viewed as restrictive in comparison to the homological approach as it assumes at least $C^2$ differentiability, and rules many of the examples of the previous section from consideration, regarding them as `degenerate'. In the Morse theoretic framework, a critical point is said to be degenerate if $\text{det}(H_f)= 0$, and non-degenerate otherwise. A non-degenerate critical point is also known as a Morse point. Note that many interesting critical points from the previous section are not Morse and thus cannot be analyzed using Morse theory (see, for instance, the Monkey saddle in Fig. \ref{fig:morse_index} or the Peano surface of Example \ref{ex:peano}). As $\text{Ind}^H(\nabla f, p)=0$ implies that $\text{det}(H_f(p))=0$, points of undulation are also not Morse points, and thus cannot be considered in the Morse theory framework. In sum, the Morse index can be used to categorize a smaller class of critical points than the homological index, but to a greater specificity.

Typically, Morse theory is only concerned with the class of Morse functions. A function $f$ is said to be Morse if it is $C^2$ and possesses no degenerate critical points. For a Morse function $f$, we denote the number of Morse points of index $\lambda$ as $N^M_\lambda(f)$\nomenclature{$N^M_\lambda(f)$}{Number of critical points of a function $f$ with Morse index $\lambda$}. 

We note here that the homological index can also be computed with relative ease for a Morse function. Specifically, if $f : \mathbb{R}^D \to \mathbb{R}$ is Morse with a non-degenerate critical point at $p$, and is convex in $j$ orthogonal directions and concave in the remaining $D - j$, then $\text{Ind}^H(\nabla f, p) = (-1)^j$. This fact can be shown by applying the Morse Lemma (c.f.\ Appendix~\ref{app:morse}), which allows $f$ to be written locally as a quadratic form with $D-j$ negative and $j$ positive terms. The normalized gradient map $\sfrac{\nabla f}{\|\nabla f\|}$ then corresponds to a reflection across $D-j$ coordinates, which can be shown to have degree $(-1)^j$ using standard results such as Propositions~13.25 and~13.27(c) of~\cite{lee2010topological}.

%

\section{Theory}\label{sec:theory}

Our aim throughout this section is to determine minimal regularity conditions under which $\{N_\lambda^H\}, \{N_\lambda^M\}, N_C, N_M, N_m$ and $N_S$ of $f_n$, converge to the corresponding quantities of $f$. In Section \ref{sec:C0}, we begin this process by considering some examples in which these quantities do not converge, given $f_n\xrightarrow{C^0} f$. Following this, in Section \ref{sec:homology}, we consider what happens to critical points when we assume that $f_n\xrightarrow{C^1} f$. As we shall see, given an additional assumption concerning the \textit{`critical point resolution'} of $f_n$, there is much that can stated in this context. In particular, we are able to make convergence statements about the number of local maxima, minima and saddles of $f_n$ converging to those of $f$. Finally, in Section \ref{sec:morse}, we consider an altogether different approach, using Morse theory to show that, if $f$ is Morse and $f_n\xrightarrow{C^2}f$, then $N_\lambda^M(f)$ converges for all $\lambda$.

{\centering
    \includegraphics[width=0.8\textwidth]{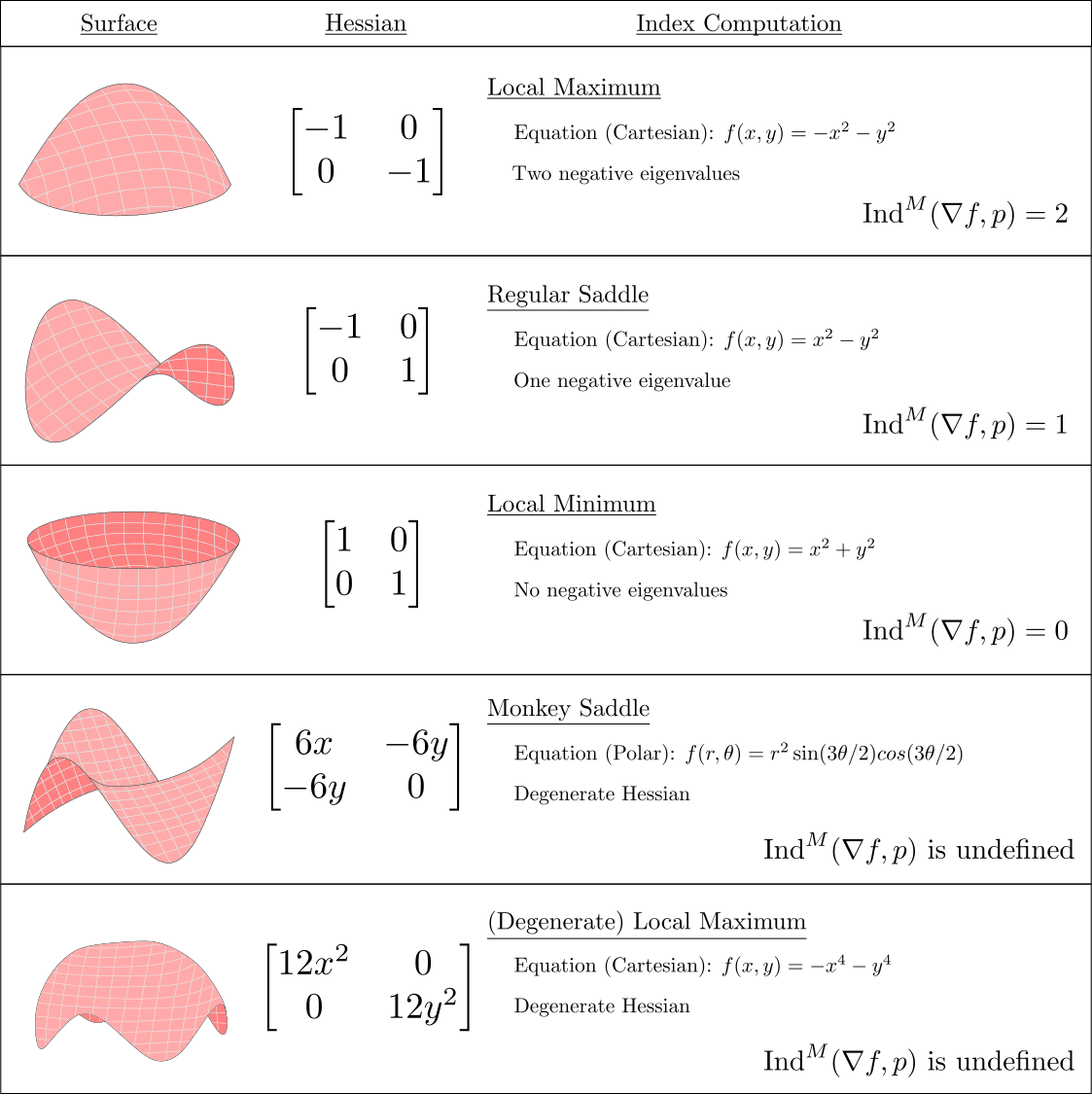}
    \captionof{figure}{Illustration of the Morse index for various Morse points. In each example, a surface plot is shown in red alongside the Hessian matrix of the function. To compute the Morse index, we count the number of negative eigenvalues of the Hessian. The final two examples illustrate settings in which the Morse index cannot be computed (but the homology index can).}
    \label{fig:morse_index}}

\subsection{Uniform Convergence}\label{sec:C0}

Suppose $S$ is a topological manifold with boundary. We begin by noting that without at least $C^1$ convergence, very little may be said in order to relate the topology of $f_n$ to that of $f$, as the following examples illustrate. 

\begin{examplebox} 
    \begin{example}\label{ex:C0}
        If $f_n:S\rightarrow \mathbb{R}$ is $C^0$ convergent to $f:S\rightarrow \mathbb{R}$, then $N_C(f_n)$ may tend to a limit larger or smaller than $N_C(f)$, or not converge at all.
        
{\centering
    \includegraphics[width=\textwidth]{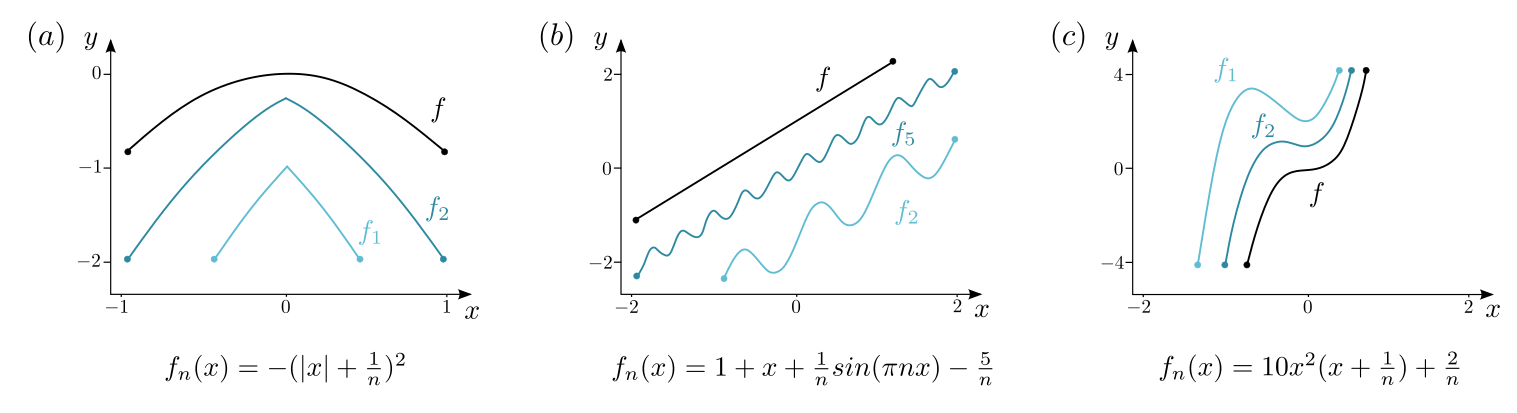}
    \captionof{figure}{Examples in which (a) $f_n\xrightarrow{C^0} f$ and $f$ is Morse, but $N_C(f_n)$ is undefined for all $n$, (b) $f_n\xrightarrow{C^0} f$, each $f_n$ is Morse and $f$ is Morse, but $N_C(f_n)\rightarrow \infty$ whereas $N_C(f)=0$, and (c) $f_n\xrightarrow{C^1} f$, each $f_n$ is Morse and $f$ is Morse, but $N_C(f_n)=2$ for all $n$, whilst $N_C(f)=1$. In the case of (b), it is worth noting that the number of critical points does not converge on any subset of $S=[-2,2]$ either.}
    \label{fig:basicexamples}}
        
    \end{example}
    \vspace*{0.1cm}
\end{examplebox}
Given these examples, it is clear that approximating topological features of $f$ using $f_n$ is not always sensible. It may still be hoped, though, that the number of local maxima and minima of $f_n$ may somehow be related to those of $f$, under mild conditions.

To explore this idea further, for this section only, we shall define a point $p\in \text{Int}(S)$ to be an improper local maxima of $f$ if it has a neighborhood over which $f\leq f(p)$, with the analogous definition holding for improper local minima. We let the number of improper maxima and minima be denoted $N_{IM}(f)$ and $N_{Im}(f)$, respectively. Panels $(b)$ and $(c)$ of Fig. \ref{fig:basicexamples} clearly demonstrate that $N_{IM}(f_n)$ does not necessarily converge to $N_{IM}(f)$ when $f_n\xrightarrow{C^0}f$ as critical points can become arbitrarily close to one another, but it may be hoped that eventually $N_{IM}(f_n)\geq N_{IM}(f)$ in some sense. This idea is formalized in the following lemma.
\begin{mdframed}
    \begin{lemma}[Improper Maxima and Minima Bound]\label{lemma:max_bdd}
        Suppose $f_n\xrightarrow{C^0}f$ and 
        that $f$ has isolated improper maxima and minima. Then:
        \begin{equation}\nonumber
            \liminf_{n\rightarrow \infty} N_{IM}(f_n)\geq N_{IM}(f) \quad \text{ and }\quad \liminf_{n\rightarrow \infty} N_{Im}(f_n)\geq N_{Im}(f).
        \end{equation}
    \end{lemma}
    \vspace{0.1cm}
\end{mdframed}
\begin{proof}
    We show the result for improper maxima, with the case of improper minima following using identical logic. As the improper maxima are isolated, by compactness of $S$ $N_{IM}(f)$ is finite. Choose a single improper maximum, $p$, of $f$ and let $B_{\epsilon_p}(p)$ be a neighborhood of $p$ containing no other improper maxima, where $\epsilon_p$ is sufficiently small that $f\leq f(p)$ in $B_{\epsilon_p}(p)$. Without loss of generality, assume that the collection of sets $\{B_{\epsilon_p}(p)\}$, indexed by improper maxima $p$, are disjoint.
    
    Suppose there exists $s^* \in \text{Int}(B_{\epsilon_p}(p))\setminus\{p\}$ such that $f(s^*)=f(p)$. Then, by the construction of $B_{\epsilon_p}(p)$, $\text{Int}(B_{\epsilon_p}(p))$ is a neighbourhood of $s^*$ over which $f\leq f(s^*)$. However, this means that $s^*$ is also an improper maxima of $f$, which contradicts the definition of $B_{\epsilon_p}(p)$. Thus for all $s\in \text{Int}(B_{\epsilon_p}(p))\setminus\{p\}$, $f(s)<f(p)$. Shrinking the value of $\epsilon_p$ if necessary, assume now that for all $s\in B_{\epsilon_p}(p)\setminus\{p\}$, $f(s)<f(p)$.
    
    It follows that $c^*:=\sup_{s\in\partial B_{\epsilon_p}(p)}f(s)<f(p)$. Choose positive $\delta< f(p)-c^*$ and assume $n$ is large enough that $||f-f_n||_\infty<\delta/2$. It now follows that $\sup_{s \in \partial B_{\epsilon_p}(p)}f_n(s)<c^* + \delta/2 < f(p)- \delta/2 < f_n(p)$. Now, define $p_n:=\arg\sup_{s\in B_{\epsilon_p}(p)}f_n(s)$. As $f_n(p_n)\geq f_n(p)> \sup_{s \in \partial B_{\epsilon_p}(p)}f_n(s)$ it follows that $p_n\not\in \partial B_{\epsilon_p}(p)$ and thus $p_n$ is an improper local maxima of $f_n$. Noting the finiteness of $N_{IM}(f)$, it follows that for each improper local maxima of $f$, $p$, there is eventually a corresponding improper local maxima of $f_n$, $p_n$. The result now follows.
\end{proof}

It is natural to wonder whether additional conditions can be placed on the distance between improper maxima and minima of $f_n$, in order to rule out examples such as $(b)$ and $(c)$ of Fig. \ref{fig:basicexamples}, in the hope that a `limsup' statement similar to Lemma \ref{lemma:max_bdd} may be derived. However, as the next example shows, without convergent first order derivatives, such an approach is not successful.
\begin{examplebox}
    \begin{example}\label{example:singlemax}
        Let $f_n:S\rightarrow \mathbb{R}$ and $f:S\rightarrow \mathbb{R}$ be $C^\infty$ and Morse and suppose that, for all $n$, $f_n$ possesses a single isolated local maxima in $\text{Int}(S)$ with no other critical points. Suppose that $f_n\xrightarrow{C^0}f$. Then it is still possible that $f$ possesses no critical points at all, and thus $N_{IM}(f)=N_M(f)=0<N_{IM}(f_n)=N_M(f_n)=1$ for all $n$. 
        
        {\centering
            \includegraphics[width=\textwidth]{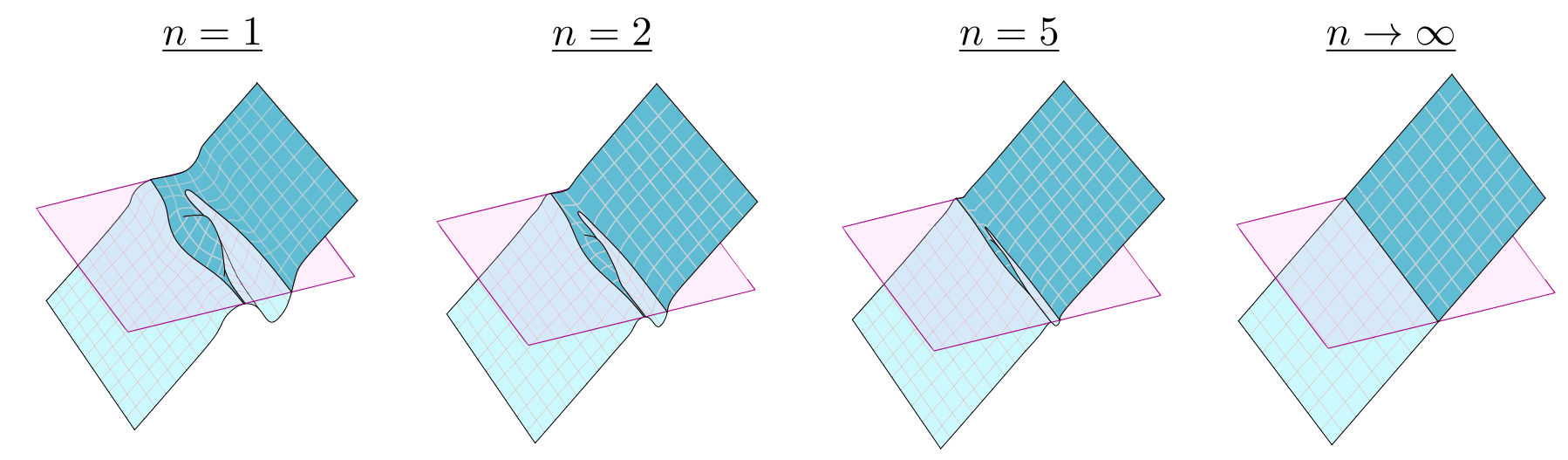}
            \captionof{figure}{A sequence of functions $f_n\xrightarrow{C^0}f$ such that each function is $C^\infty$ and Morse and every $f_n$ has a single isolated local maxima, but $f$ has no critical points at all. }
            \label{fig:singlemax}}
        
        \vspace{0.2cm}
        \noindent The function in Fig. \ref{fig:singlemax} is given by $f_n(x,y)=g(x,ny)/n$ where $g$ is defined below:
        \begin{equation}\nonumber
            g(x,y):=\begin{cases}
                y & y \in (-\infty,-1],\\
                (1-b(y))y+b(y)e^{-x^2} & y \in [-1,0],\\
                (1-b(y))s(x)+b(y)e^{-x^2} & y \in [0,1],\\
                b(y-1)s(x)+(1-b(y-1))(y-1) & y \in [1,2],\\
                y-1 & y \in [2,\infty).\\
            \end{cases}
        \end{equation}
        Here, the smooth bump function $b:\mathbb{R}^m \rightarrow \mathbb{R}$ is defined, for $m\geq 1$, by $b(s)=\text{exp}(1-\frac{1}{1-|s|^2})$ if $|s|<1$ and $0$ otherwise and the transition function $s:\mathbb{R}\rightarrow \mathbb{R}$ is given by $s(x)=-e^x/(e^x+e^{1-x})$.
    \end{example}
\end{examplebox}

To guarantee that the topology of $f_n$ eventually resembles that of $f$, it is clear that stronger conditions must be assumed. However, as illustrated by Fig. \ref{fig:basicexamples} (c), even $C^1$ convergence alone is not sufficient.

\subsection{Derivative Convergence and the Poincar\'{e}-Hopf Theorem}\label{sec:homology}

We now consider the setting in which $f_n\xrightarrow{C^1}f$. In this section, we will make use of two lesser known variants of two well-known theorems. The first, and most important, is the Poincar\'{e}-Hopf theorem, which we use to show that $f_n$ does indeed (eventually) have critical points located near those of $f$. Although traditionally stated for $C^1$ manifolds, the variant of Poincar\'{e}-Hopf that we employ here is based on the extension to $C^1$ manifolds with boundary provided by \cite{jubin2009generalized}. The second is the mountain pass theorem, which we use to distinguish local maxima, minima and saddles from one another. Although variants of this theorem are well-known, the version we employ is unconventional in that allows for $S$ to be a manifold with boundary and thus does not make use of the standard Palais-Smale conditions (c.f. \cite{Jabri_2003} Sections 10.2 and 17.2).

To begin, we first make a simplifying assumption.
\begin{examplebox}
\begin{assumption}[Critical Points on the Boundary]\label{assump:no_bdry_crits} $f:S\rightarrow \mathbb{R}$ has no critical points in $\partial S$. That is, $\partial S \cap Z(\nabla f) = \emptyset$.
    \vspace{0.3cm}
\end{assumption}
\end{examplebox}
\begin{remark}
    Without the above assumption, little can be said about the convergence of critical points of $f_n$ in relation to $f$. To understand why, consider $S=[0,1]$, $f(x)=-x^2$ and $f_n(s)=-(x-\frac{1}{n})^2$. Both $f$ and $f_n$ possess one critical point in $\mathbb{R}$, but $f_n$ never has a critical point in $S$. In this case, even though $f$ has a critical point $p \in \partial S$ and, for all $n$, $f_n$ has a corresponding critical point lying arbitrary close to $p$, this point always lies outside of the domain of interest $S$, and thus $N_C(f_n)\not\rightarrow N_C(f)$.\\
\\
  \noindent As $\partial S$ is compact, Assumption \ref{assump:no_bdry_crits} implies that on $\partial S$, $|\nabla f|$ is bounded below by a positive constant. If we additionally assume that $f_n\xrightarrow{C^1} f$, then the uniform convergence of $\nabla f_n$ implies that, for sufficiently large $n$, $|\nabla f_n|$ is non-zero on $\partial S$. Thus without loss of generality, whenever we assume (at least) $C^1$ convergence in the following, we shall assume that $n$ is large enough that all critical points of $f_n$ lie inside Int$(S)$.
\end{remark}

Given Assumption \ref{assump:no_bdry_crits} and $C^1$ convergence, we can now simplify the problem to the setting in which $S$ is a compact subset of $\mathbb{R}^D$. This is made explicit by the following lemma.
\begin{mdframed}
\begin{lemma}[Critical Neighborhoods in $\mathbb{R}^D$]\label{lemma:mfld_to_R}
    Suppose $f_n\xrightarrow{C^0}f$, $f_n$ is differentiable, $f$ is continuously differentiable and $|\nabla f_n-\nabla f|\rightarrow 0$ uniformly. Further, suppose that Assumption \ref{assump:no_bdry_crits} holds and, for each critical point $p$ of $f$ let $(U_p, \phi_p)$\nomenclature{$(U,\phi)$}{Chart on $S$} be a chart such that $p\in U_p$ and $\phi_p(p)=0\in \phi_p(U_p)$. Then, for $n$ large enough, for each critical point $p$ there exists $\eta_p>0$, such that $f_n$ has no critical points in $S\setminus \cup_p \phi_p^{-1}(B_{\eta_p})$.
    \vspace{0.2cm}
\end{lemma}
\end{mdframed}
\begin{remark}
    The below proof follows the same strategy as that of Lemma A.1 of the supplement to \cite{Cheng2017} and Theorem 3.2 of \cite{davenport2022confidenceregionslocationpeaks}. 
\end{remark}
\begin{proof}
    For each critical point $p$, chose $\eta_p>0$ such that $B_{\eta_p}\subset \phi_p(U_p)$. As the set $V:=S\setminus \cup_p \text{Int}(\phi_p^{-1}(B_{\eta_p}))$ is a closed subset of the compact space $S$, it is compact. As it is compact and $\nabla f$ is continuous, $|\nabla f|$ attains it's infinimum on $V$. If $\inf_{s \in V}|\nabla f(s)|=0$ then $f$ would have a critical point outside of $V$, contradicting the definition of $V$. It follows that $\inf_{s \in V}|\nabla f(s)|> \epsilon$ for some $\epsilon > 0$. By the uniform convergence of $\nabla f_n\xrightarrow{C^1} \nabla f$, for $n$ large enough, $|\nabla f_n-\nabla f|<\frac{\epsilon}{2}$ and thus $\inf_{s \in V}|\nabla f_n(s)|> \frac{\epsilon}{2}$. The result now follows.
\end{proof}
    Note that, if the critical points of $f$ are isolated, then $\{\eta_p\}_p$ can be chosen so that $\cup_p\phi^{-1}(B_{\eta_p})$ is diffeomorphic to a disjoint collection of balls in $\mathbb{R}^D$, each of which only contains a single critical point of $f$. By compactness and disjointness, it follows that this union must be finite (i.e. $f$ has finitely many critical points). In other words, the above lemma tells us that if we wish to show that the number of critical points of $f_n$ converge to those of $f$, we need only consider the behavior of $f_n$ and $f$ on each of a finite union of disjoint compact subsets of $\mathbb{R}^D$.
    
    Consequently, unless otherwise specified, we now restrict our attention to functions $f_n$ and $f$ defined on $S := B_\eta \subset \phi(U_p) \subset \mathbb{R}^D$ for some $\eta > 0$, where $p$ is the only critical point of $f$ within $B_\eta$. \nomenclature{$\eta$}{Radius of $S$ when $S$ is assumed to be a ball} The generalization to the arbitrary compact manifold with boundary case follows easily by considering the pushforwards of $f_n$ and $f$ and performing induction over the finitely many critical points.

Our aim is to show the desired convergences by placing minimal conditions on the functions' derivatives. Specifically, in this section, we explore the consequences of assuming $f_n\xrightarrow{C^1}f$ with a lower bound placed on the distance between critical points of $f_n$. The condition we consider is given explicitly below:
\begin{examplebox}
    \begin{assumption}[Critical Point Resolution - Topological Version]\label{assumption:crit_tplgcl} We say that $f_n$ has \textit{bounded critical point resolution} if, for every convergent sequence of critical points $p_n$ of $f_n$ with limit $p^*$, there exists a neighborhood of $p^*$, $U_{p^*}$, such that for sufficiently large $n$, $U_{p^*}$ contains at most one critical point of $f_n$.
    \end{assumption}
    \vspace{0.3cm}
\end{examplebox}
\begin{remark}
    When $S$ is additionally a metric space, the condition can be reformulated to the below.
\end{remark}

\begin{examplebox}
    \begin{assumption}[Critical Point Resolution - Metric Version]\label{assumption:crit} Suppose that each $f_n$ is differentiable and $S$ is a metric space. We say that the sequence $\{f_n\}$ has \textit{bounded critical point resolution} if the following condition holds:\nomenclature{$R(\{ f_n\})$}{Critical point resolution of the sequence $\{f_n\}$}\\
    \begin{equation}\nonumber
    \quad R(\{ f_n\}):=\liminf_{n\rightarrow\infty}\inf_{\substack{p_1,p_2\in Z(\nabla f_n)\\ p_1\neq p_2}}|p_1-p_2|>0. 
    \end{equation}
    \end{assumption}
    \vspace{0.3cm}
\end{examplebox}
\begin{remark}
    Informally, the above states that as $n$ grows large, the distance between critical points (i.e. the `resolution', $R$) of $f_n$ is bounded below. It is worth highlighting that, whilst all other conditions considered in this work have considered the regularity of both $\{f_n\}$ and $f$, this approach has the advantage of being given purely in terms of the behavior of $\{f_n\}$.
\end{remark}

In light of Lemma \ref{lemma:mfld_to_R}, we take Assumption \ref{assumption:crit} to hold throughout the remainder of this section. This is possible as it can easily be shown that if Assumption \ref{assumption:crit_tplgcl} holds for an arbitrary compact manifold with boundary $S$, then the metric variant holds on the finite disjoint union $\cup_p \phi_p(U_p)$. To illustrate why we might be interested in bounding the critical point resolution of $f_n$ consider the following example.
\begin{examplebox}
    \begin{example}\label{example:example6}
        Let $f_n:S\rightarrow \mathbb{R}$ be Morse and $f:S\rightarrow \mathbb{R}$ be Morse. Suppose that $f_n\xrightarrow{C^1}f$. Then it is still possible that $N_M(f_n)$ does not converge to $N_M(f)$.
        
        {\centering
            \includegraphics[width=0.5\textwidth]{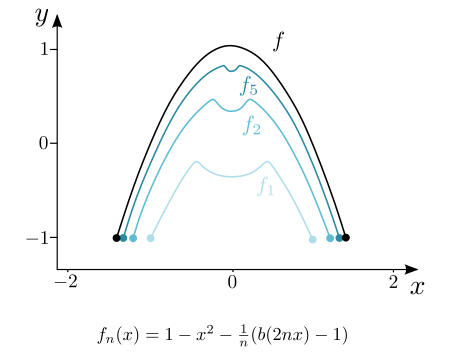}
            \captionof{figure}{A sequence of functions $f_n\xrightarrow{C^1}f$ such that all functions are Morse, but every $f_n$ has two local maxima, whilst the limit $f$, given by $f(x)=1-x^2$, has only one. Here, the smooth bump function $b$ is defined as in Example \ref{example:singlemax}. In this case, it is clear that $R(\{ f_n\})=0$ and thus Assumption \ref{assumption:crit} is not satisfied.}
            \label{fig:example6}}
    \end{example}
\end{examplebox}

As shown by Example \ref{example:example6}, $C^1$ convergence is not sufficient to guarantee that the number of critical points converges. However, under the assumption of bounded critical point resolution, the number of critical points of $f$ does act as a limiting upper bound for that of $f_n$. This statement is made explicit below:
\begin{mdframed}
    \begin{thm}[Critical Point Upper Bound]\label{lem:upper_bdd}
        Suppose that $S$ is a compact $C^1$ manifold with boundary, $f_n\xrightarrow{C^1}f$ where $f$ has isolated critical points and Assumptions \ref{assump:no_bdry_crits} and \ref{assumption:crit} hold. 
    Then: 
    \begin{equation}\nonumber
        \limsup_{n\rightarrow\infty}N_C(f_n)\leq N_C(f)
    \end{equation}
    \end{thm}
    \vspace{0.3cm}
\end{mdframed}

\begin{proof}
    By Assumption \ref{assumption:crit}, we can find $N \in \mathbb{N}$ and $\epsilon > 0$ such that if $n \geq N$, all critical points of $f_n$ are at least $\epsilon$ distance apart. By identical logic to that employed in the proof of Lemma \ref{lemma:mfld_to_R}, it also follows that if $N$ is large enough then, for $n\geq N$, all critical points of $f_n$ are within distance $\epsilon/2$ of a critical point of $f$. Thus, inside each $\epsilon/2$-ball around a critical point of $f$, there can be at most one critical point of $f_n$, and outside the union of these balls there are no critical points of $f_n$. Thus, $N_C(f_n)\leq N_C(f)$. Applying limit suprema we see that $\limsup_{n\rightarrow\infty}N_C(f_n)\leq N_C(f)$. 
\end{proof}
Although the above lemma shows that under the given conditions the number of critical points is bounded, it says nothing about the types of critical points under consideration. It might be hoped that, in this setting, something may be said about the convergence of the homological or Morse indices, whenever they are defined. However, as the following examples illustrate, even if $f_n$ and $f$ are $C^2$, $f_n\xrightarrow{C^1} f$ and there is a one-to-one relationship between their critical points, unusual behavior can arise for the Morse index.

\begin{examplebox}
    \begin{example}
        Let $f_n:S\rightarrow \mathbb{R}$ be Morse and $f:S\rightarrow \mathbb{R}$ be Morse. Suppose that $f_n\xrightarrow{C^1}f$, each $f_n$ has a single critical point, $p_n$, and $f$ has a single critical point $p$. Then it is still possible that $H_n(p_n)$ does not converge to $H(p)$.
    \end{example}
{\centering
    \includegraphics[width=\textwidth]{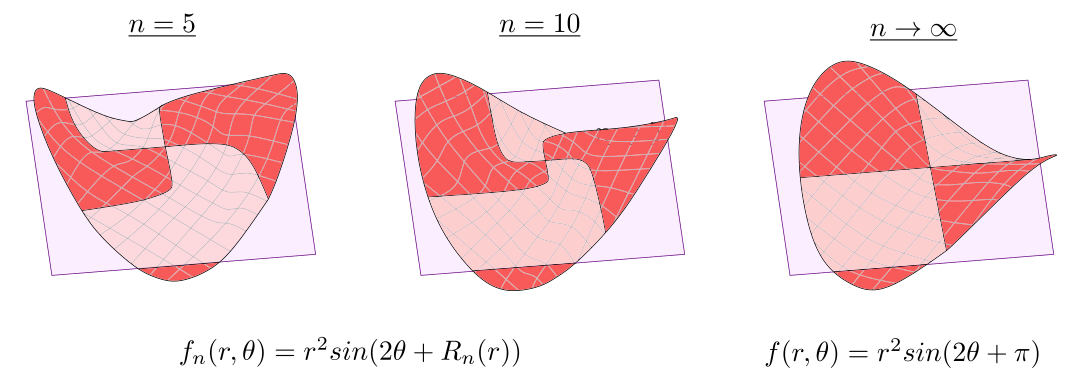}
    \captionof{figure}{A sequence of functions $f_n\xrightarrow{C^1}f$ such that each function is Morse and has a single isolated critical point at the origin, but the Hessians do not converge. In fact, for all vectors $\vec{v}$, if $f$ is concave in the direction of $\vec{v}$ at the origin, then $f_n$ is convex in that direction, and vice versa. In this example, the function $R_n:\mathbb{R}_{\geq 0}\rightarrow \mathbb{R}$, which affects a rotation of the surface, is given by $R_n(x):=\pi \text{exp}(\frac{-1}{nx-1})$ if $x> \frac{1}{n}$ and $0$ otherwise.}
    \label{fig:twisted_saddle}}
\end{examplebox}

\begin{examplebox}
    \begin{example}\label{ex:peano}
        Let $f_n:S\rightarrow \mathbb{R}$ be Morse and $f:S\rightarrow \mathbb{R}$ be $C^2$. Suppose that $f_n\xrightarrow{C^2}f$, each $f_n$ has a single critical point, $p_n$, and $f$ has a single critical point $p$. It is still possible that $f$ is not Morse and that the number of negative eigenvalues of $H_n$ does not converge to that of $H$.
        
{\centering
    \includegraphics[width=\textwidth]{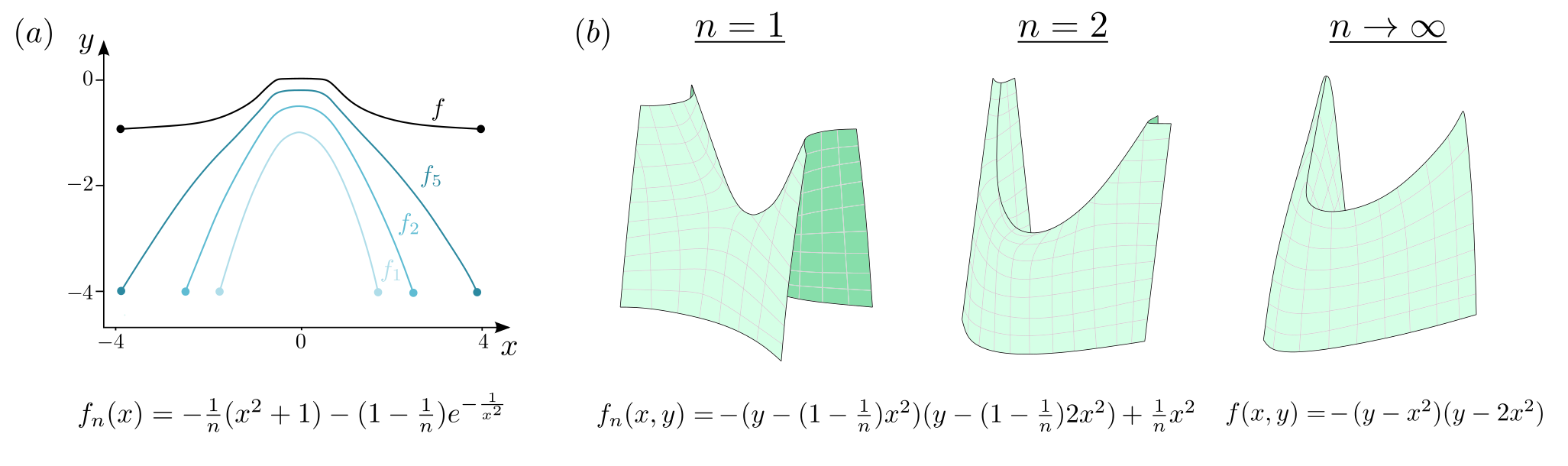}
    \captionof{figure}{Two sequences of functions $f_n\xrightarrow{C^2}f$ such that $f$ possesses a degenerate critical point at the origin, and is thus not Morse but, for all $n$, $f_n$ is Morse. In panel $(a)$, $H_n$ has a single positive eigenvalue for all $n$ (i.e. $f_n''>0$), but as the limit is given by $f(x):=-e^{-\frac{1}{x^2}}$, we have that $H=[0]$ and thus has no positive eigenvalues. Similarly in panel $(b)$, $H_n$ has one negative and one positive eigenvalue for all $n$, but the limit has no non-zero eigenvalues. In this case, $f$ is the infamous Peano surface, which has negative curvature along every straight path from the origin, but is not a local maxima.}
    \label{fig:peanoplusex4}}
    \end{example}
\end{examplebox}

It is clear from the above examples that even if we assume that the number of critical points of $f_n$ equals that of $f$, both $f_n$ and $f$ are $C^2$ and that $f_n\xrightarrow{C^2}f$, this is not sufficient to guarantee that the Morse indices of critical points of $f_n$ provide information about those of $f$. However, it may still be hoped that, given $f_n\xrightarrow{C^1}f$, only mild additional assumptions are required in order to make statements about the maxima, minima and saddle points of $f_n$ and $f$. It turns out that the condition required is Assumption \ref{assumption:crit}, the consequences of which we explore in the following theorems. 
\begin{mdframed}
    \begin{thm}[Homological Index Convergence]\label{thm:homological_index}
    Suppose that $S$ is a compact $C^1$ manifold with boundary, $f_n\xrightarrow{C^1}f$ where $f$ has isolated critical points and Assumptions \ref{assump:no_bdry_crits} and \ref{assumption:crit} hold. Then, for any $\lambda\neq 0$, $N^H_\lambda(f_n)\rightarrow N^H_\lambda(f)$ as $n\rightarrow\infty$, while $\limsup_{n\rightarrow \infty} N^H_0(f_n)= N^H_0(f)$.
    \end{thm}
    \vspace{0.3cm}
\end{mdframed}

    \begin{proof}
    Without loss of generality, assume $p$ is the critical point of $f$, $S=B_\eta(p)$ and $\epsilon<\max(R(\{ f_n\}),\eta)/4$. Then, by the definition of critical point resolution and Assumption \ref{assumption:crit} (c.f. Assumption \ref{assumption:crit_tplgcl} also, which is equivalent in this setting), eventually $B_{4\epsilon}(p)$ contains at most one critical point of $f_n$. By restricting attention to the compact manifold with boundary $B:=B_{4\epsilon}(p)$, we now define a partition of unity $\{\phi_1,\phi_2\}$ on the open cover $\{\text{Int}(B_{3\epsilon}(p)), B\setminus B_{2\epsilon}(p)\}$ (see Fig. \ref{fig:hom_index_proof}). Note that by construction, $B\setminus B_{2\epsilon}(p)$ contains no critical points of $f$. 
    
    We now define $M:=\sup_{s\in B}(\max(|\nabla\phi_1(s)|,1))$. By the definition of uniform convergence, we can assume that $n$ is large enough that:
    \begin{equation}\nonumber
        \max\bigg(||f-f_n||_\infty,||\nabla f_n-\nabla f||_\infty\bigg) <\delta:= \frac{1}{3M} \inf_{s \in B\setminus B_{2\epsilon}(p)}|\nabla f(s)|
    \end{equation}
    It is immediate that, for $s \in B\setminus B_{2\epsilon}(p)$, we have
    \begin{equation}\nonumber
    \begin{aligned}
        |\nabla f_n(s)| & \geq \big||\nabla f_n(s)-\nabla f(s)|-|\nabla f(s)|\big| \\
        & > \bigg(1-\frac{1}{3M}\bigg)|\nabla f(s)| \\
        & > 0,
    \end{aligned}
    \end{equation}
    where the first inequality follows by the reverse triangle inequality and the second by the definition of $\delta$. Thus, $f_n$ has no critical points in $B\setminus B_{2\epsilon}(p)$.

    Next, we let $\tilde{f}_n=\phi_1 \cdot f_n + \phi_2 \cdot f$, noting that $\tilde{f}_n= f_n$ on $B_{2\epsilon}(p)$ and $\tilde{f}_n= f$ on $\overline{B\setminus B_{3\epsilon}(p)}$. By the definition of partition of unity, it trivially follows that
    \begin{equation}\nonumber
        \nabla \tilde{f}_n = \nabla \bigg(f + \phi_1 \cdot (f_n-f)\bigg)= \nabla f + \nabla \phi_1 \cdot (f_n-f) + \phi_1 \cdot \nabla (f_n-f).
    \end{equation}
    And thus, by the reverse triangle inequality, for $s \in B\setminus B_{2\epsilon}(p)$;
    \begin{equation}\nonumber
    \begin{split}
        |\nabla \tilde{f}_n(s)| & \geq |\nabla f(s)| - |\nabla \phi_1(s)|\cdot |f_n(s)-f(s)| - |\phi_1(s)| \cdot |\nabla f_n(s) - \nabla f(s)| \\
        & > |\nabla f(s)| - 2 M\delta \\ 
        & = |\nabla f(s)| - \frac{2}{3} \inf_{s \in B\setminus B_{2\epsilon}(p)}|\nabla f(s)| \\ 
        & > \frac{1}{3}|\nabla f(s)| > 0, \\ 
    \end{split}
    \end{equation}
    
{\begin{floatingfigure}[r]{0.48\textwidth}\centering\vspace{-2cm}
    {\centering\includegraphics[width=0.32\textwidth]{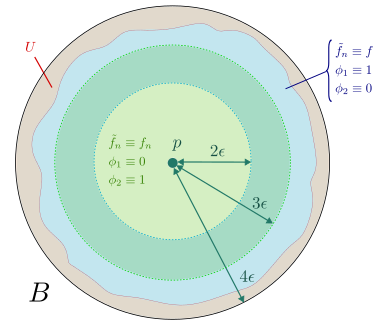}
    \captionof{figure}{The partition of unity used in the proof of Theorem \ref{thm:homological_index}. In the light green region, $\tilde{f}_n=f_n$, while in the blue region $\tilde{f}_n=f$. Also shown in light red, on top of the blue region, is a potential choice of collar, $U$, for $B$.}\label{fig:hom_index_proof}}
\end{floatingfigure}}

\noindent where the second inequality follows from the definitions of $\delta$ and $M$ and the last inequality follows from the fact that $B\setminus B_{2\epsilon}(p)$ contains no critical points of $f$. Thus, $\tilde{f}_n$ has no critical point in $B\setminus B_{2\epsilon}(p)$. Further, as $\tilde{f}_n\equiv f_n$ on $B_{2\epsilon}(p)$ and $2\epsilon < R(\{ f_n\})$, it follows that $B_{2\epsilon}(p)$ contains at most one critical point of $\tilde{f}_n$ which, if existent, is also a critical point of $f_n$ with the same homological index. Applying the generalized Poincar\'{e}-Hopf theorem, Theorem \ref{thm:poincarehopf}, to $B$, we now have that:
\begin{equation}\nonumber
    \text{Ind}^H_\partial(\nabla f) + \text{Ind}^H_\circ(\nabla f) = \text{Ind}^H_\partial(\nabla \tilde{f}_n) + \text{Ind}^H_\circ(\nabla \tilde{f}_n) = \begin{cases}
    1 & \text{if }D \text{ is even,}\\
    0 & \text{if }D \text{ is odd.}
    \end{cases}
\end{equation}
However, by choosing a collar of $B$ with domain $U$ contained in $\overline{B\setminus B_{3\epsilon}(p)}$ and noting that $\tilde{f}_n= f$ on $B\setminus B_{3\epsilon}(p)$, we can see that $\text{Ind}^H_\partial(\nabla \tilde{f}_n)=\text{Ind}^H_\partial(\nabla f)$. Thus, we have $\text{Ind}^H_\circ(\nabla f) = \text{Ind}^H_\circ(\nabla \tilde{f}_n)$.
As $B_{2\epsilon}(p)$ contains exactly one critical point of $f$, namely $p$, we have that $\text{Ind}^H_\circ(\nabla f)=\text{Ind}^H(\nabla f, p)$. Recalling that $B\setminus B_{2\epsilon}(p)$ contains no critical points of $f_n$, we now see that as $B_{2\epsilon}(p)$ contains at most one critical point of $\tilde{f}_n$, which must also be a critical point of $f_n$ with the same index, we have:
\begin{equation}\nonumber
    \text{Ind}^H_\circ(\nabla \tilde{f}_n) := \begin{cases}
        \text{Ind}^H(\nabla f_n, p_n) & \text{if $f_n$ has a critical point $p_n$ in $B_{4\epsilon}(p)$,} \\
        0 & \text{otherwise}.
    \end{cases}
\end{equation}
Suppose it were the case that $p$ is not an undulation point, i.e. $\text{Ind}^H(\nabla f, p)\neq 0$, and $f_n$ has no critical point in $B_{4\epsilon}(p)$. Then we would have:
\begin{equation}\label{eq:hom_index_equality}
    0\neq \text{Ind}^H(\nabla f, p) = \text{Ind}^H_\circ(\nabla f)=\text{Ind}^H_\circ(\nabla \tilde{f}_n) = 0,
\end{equation}
causing a contradiction. Thus, if $p$ is a critical point of homological index $\lambda\neq 0$, then, for $n$ large enough, $f_n$ must also have exactly one critical point $p_n$ which is of homological index $\lambda$. \\
\indent Now suppose that $p$ were an undulation point. By similar reasoning to the above, it follows that $B_{4\epsilon}(p)$ can contain at most one critical point of $f_n$ and that this critical point, if it exists, must have homological index zero. By employing the arguments of Lemma \ref{lemma:mfld_to_R} to generalize to arbitrary manifolds with boundary $S$ and the case in which $f$ has multiple critical points, the result now follows.\\
\end{proof}

The above theorem provides convergence guarantees about the number critical points of a given homological index. However, as the homological index of maxima, saddles and minima can agree with one another in various dimensions, the theorem only provides partial information about the convergence of $N_M, N_m$ and $N_S$. We shall shortly address this issue with Theorem \ref{thm:not_quite_morse}. Before we do so, however, we provide a short example to aid intuition with the proof of Theorem \ref{thm:not_quite_morse}, in which a sequence of isolated local maxima converges to a saddle point, and a lemma to support the proof.

\begin{examplebox}
    \begin{example}\label{example:isolated_crit}
        Let $f_n:S\rightarrow \mathbb{R}$ be Morse and $f:S\rightarrow \mathbb{R}$ be Morse. Suppose that $f_n\xrightarrow{C^2}f$ and each $f_n$ has a single isolated local maxima, $p_n$. Then it is still possible that the limit of $p_n$ is not a local maxima of $f$, but is instead a saddle point.
        
{\centering
    \includegraphics[width=0.85\textwidth]{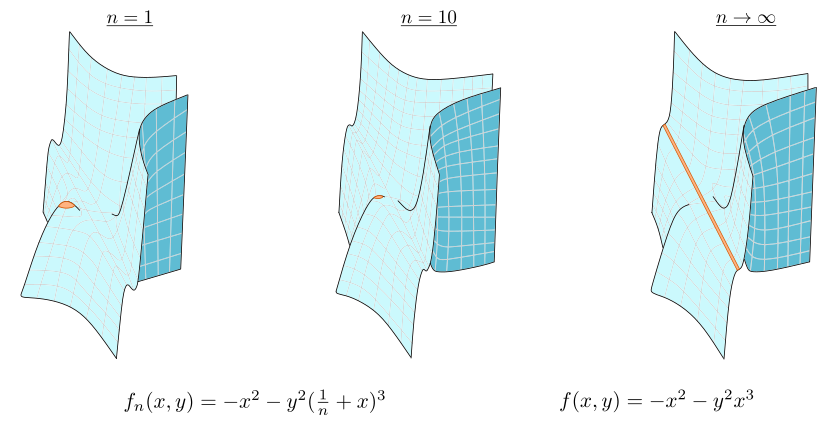}
    \captionof{figure}{ A sequence of functions $f_n\xrightarrow{C^2} f$ such that each $f_n$ has a single critical point, which is a local maxima, but the limit of these critical points is a saddle point of $f$. Highlighted in orange are the critical points of each function. Note that, while each $f_n$ has a single isolated critical point, $f$ has a single saddle point lying in the center of a continuous line of undulation points.}
    \label{fig:isolated_crit}}

    \end{example}
\end{examplebox}
A few points are worth noting about Example \ref{example:isolated_crit}. Firstly, it shows that an assumption of isolated critical points of $f_n$ does not help ensure critical point convergence. Second, the functions $\{f_n\}$ are classic illustrations of the fact that an isolated local maxima need not be a global maxima, even in the case that no other critical points exist. Thirdly, it is of interest to note that the interior homological index of each local maxima is $1$, but the interior homological index of the limiting saddle is $-1$ and the undulation points have index $0$. This discordance illustrates the care that must be taken when computing the homological index, as Definitions \ref{def:homological_index_mfld} and \ref{def:homological_index} only apply when the critical points of $f$ are isolated. 

Finally and most importantly, we note that, for each $n$ the edge of the surface looks approximately like a cubic polynomial. Specifically, the restriction of each $f_n$ to the boundary has a local minima and maxima, near to the front of the image. At these critical points, it follows that the (unrestricted) gradient of $f$ must be directly perpendicular to the boundary of the image (i.e. it points outwards for the minima, and inwards for the maxima). This is a fact which we shall exploit to reach a contradiction.

To construct the desired contradiction, we first need one more lemma. Suppose that $f\in C^1(S,\mathbb{R})$ has a single isolated critical point $p$ with $f(p)=c$. Ideally, for what follows, we would like to show that $f$ must satisfy the following condition: for any $\epsilon>0$, there exists a $\delta\in (0,\epsilon)$ such that for all $s \in \partial B_{\delta}(p) \cap f^{-1}(c)$:
\begin{equation}\label{eq:condition}
    \nabla f(s) \quad \text { and } \quad |s-p| \quad \text{are not collinear}.
\end{equation}
In other words, we would like for $\partial B_{\delta}(p)$ and $f^{-1}(c)$ to intersect \textit{transversally} (see Appendix \ref{app:transversal} for further detail and discussion of transversality). This seemingly innocuous condition holds for most practical examples of $C^1$ functions (c.f. Example \ref{example:crinkled}). However, as shown in Supplementary Material Section \ref{supp:cantor}, it is possible to construct unexpected fractal-like functions for which this condition does not hold.

\begin{examplebox}
    \begin{example}\label{example:crinkled}
Unusual examples of functions for which Condition \eqref{eq:condition} is satisfied despite exhibiting unexpected behavior. Specifically in Fig \ref{fig:crinkled} Panel $(a)$ the set $\mathcal{K}:=\{s\in f^{-1}(c):\nabla f(s) \propto |s-p|\}$ contains a circular arc, and in Panel $(b)$ $\mathcal{K}$ intersects $B_\epsilon(p)$ for countably infinitely many $\epsilon>0$. In both such cases, however, for any choice of $\epsilon > 0$, it is possible to choose a $\delta \in (0,\epsilon)$ such that the condition is satisfied.

{\centering
    \includegraphics[width=0.8\textwidth]{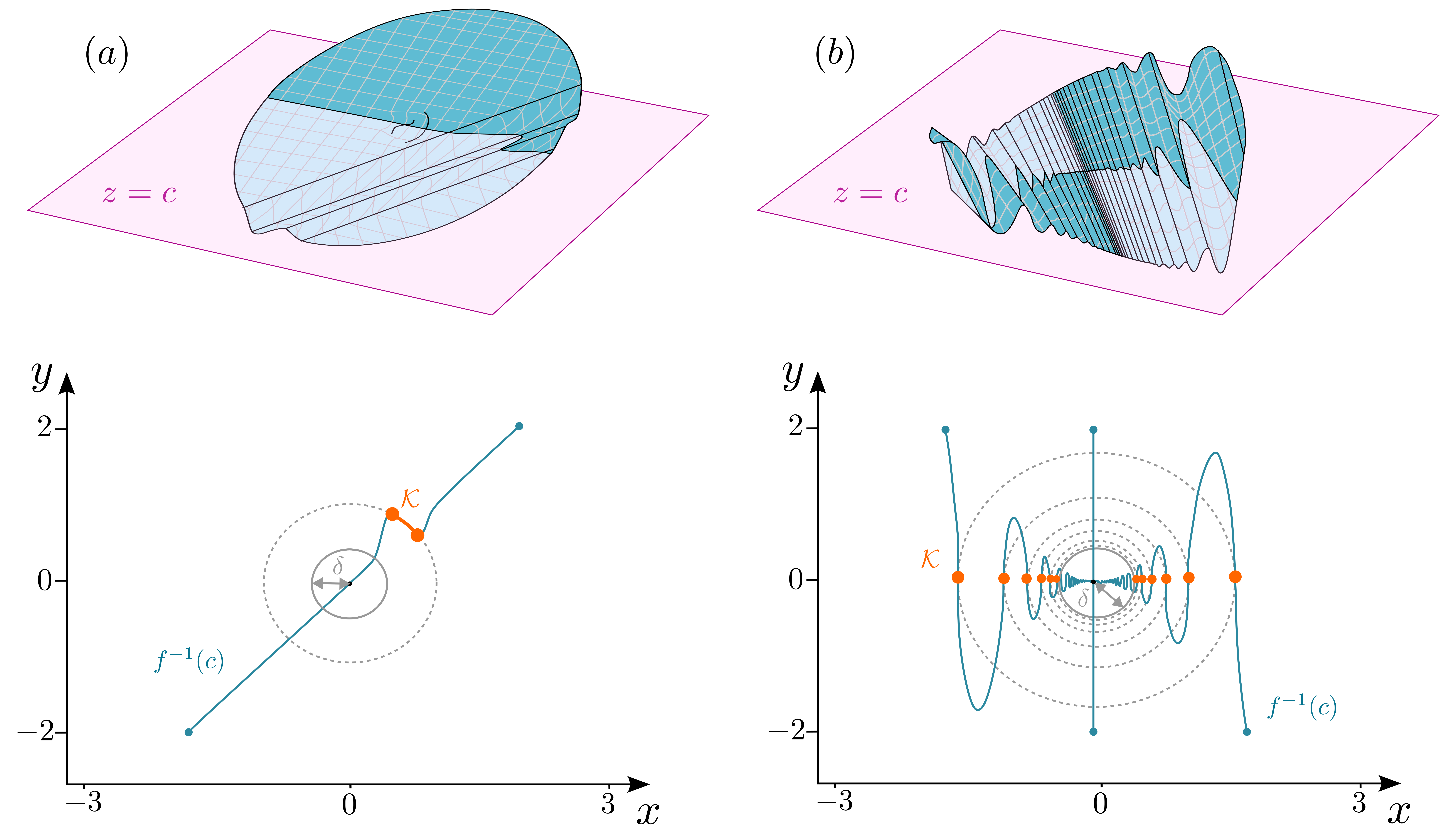}
    \captionof{figure}{Surface plots for two functions (top), alongside top down views of the contour lines $f^{-1}(c)$ (bottom). In (a), $f$ has a single isolated undulation point at the origin, and $\mathcal{K}$ consists of a single connected component. In (b) there is a single isolated saddle at the origin, surrounded by fractal-like oscillations and $\mathcal{K}$ consists of many isolated points. Equations for these surface plots are given in Supplementary Material Section \ref{supp:crinkled}. }
    \label{fig:crinkled}}
        
    \end{example}
\end{examplebox}
To overcome this issue, we make use of the following lemma. In words, this says that if $f_n\rightarrow f$ and $f$ does not satisfy condition \eqref{eq:condition}, we can construct functions $\tilde{f}_n\rightarrow \tilde{f}$ such that $\tilde{f}$ and $f$ and, for $n$ large enough, $\tilde{f}_n$ and $f_n$ have identical critical points, but $\tilde{f}$ does satisfy condition \eqref{eq:condition}.

\begin{mdframed}
    \begin{lemma}[Transversal Adjustment Lemma]\label{lem:transversal} Suppose $f_n\xrightarrow{C^1}f$ and $f$ has a single isolated critical point, $p$, and let $\epsilon > 0$. Then there exists $\epsilon > \delta > \gamma > 0$ and $C^1$ functions $\tilde{f}_n,\tilde{f}:B_\epsilon(p)\rightarrow \mathbb{R}$ such that the following conditions hold:
    \begin{itemize}
        \item[(i)] $\delta$ and $\tilde{f}$ satisfy Condition \eqref{eq:condition}.
        \item[(ii)] The restrictions of $\tilde{f}$ and $\tilde{f}_n$ to $B_{\gamma}(p)$ agree with those of $f$ and $f_n$.
        \item[(iii)] For sufficiently large $n$, $f_n$ and $\tilde{f}_n$ have no critical points in $B_\epsilon(p)\setminus \text{Int}(B_\gamma(p))$.
    \end{itemize}  
    In such settings, we shall call $\tilde{f}$ and $\tilde{f}_n$ transversal adjustments of $f$ and $f_n$, respectively.\end{lemma}
    
    \vspace{0.3cm}
\end{mdframed}

\begin{proof}
    Without loss of generality assume $p$ is the origin. Choose $\epsilon>0$ and let $\gamma=\epsilon/4$ and, making $\epsilon$ smaller if necessary, assume that $B_{\epsilon}$ is connected. As $f$ contains no critical points in $B_{\epsilon}\setminus B_{\gamma}$ by construction, we have $M:=\inf_{s\in B_{\epsilon}\setminus B_{\gamma}} |\nabla f(s)|>0$. Let $\varphi_\eta:B_\eta\rightarrow \mathbb{R}$ be a smooth bump function with non-zero support on $\text{Int}(B_\eta)$ and define $f^\eta:=\varphi_\eta *f$, where $*$ represents convolution over $B_\eta(s)$. Note that for arbitrary $s\in S$:
    \begin{equation}\nonumber 
        \begin{split}
            |\nabla f^\eta(s)-\nabla f(s)| & = \bigg|\int_{B_\eta(s)} \varphi_\eta(x-s)\bigg(\nabla f(x)-\nabla f(s)\bigg)dx\bigg| \\
             & \leq \int_{B_\eta(s)} \varphi_\eta(x-s)\big|\nabla f(x)-\nabla f(s)\big| dx\\
              & \leq \omega_\eta(\nabla f)\int_{B_\eta(s)} \varphi_\eta(x-s) dx=\omega_\eta(\nabla f)\\
        \end{split}
    \end{equation}
    where $\omega_\eta(\nabla f):=\sup_{|x-y|<\eta}|\nabla f(x)-\nabla f(y)|$ is the modulus of continuity of $\nabla f$. Similar logic shows that $|f^\eta(s)-f(s)|\leq \omega_\eta(f)$. Thus, choosing $\eta$ small enough that $2\omega_\eta(f)+\omega_\eta(\nabla f)<M/2$ we obtain for $s \in B_{\epsilon}\setminus B_{\gamma}$:
    \begin{equation}\nonumber
    \begin{split}
        |\nabla f^\eta(s)| & \geq \bigg| |\nabla f(s)| - |\nabla f(s) - \nabla f^\eta(s)|\bigg| \\
        & \geq M-M/2 = M/2 > 0.
    \end{split}
    \end{equation}
    It follows that $f^\eta$ has non-zero gradient on $B_{\epsilon}\setminus B_{\gamma}$. Now define:
    \begin{equation}\nonumber
        \tilde{f}(s):=\begin{cases}
            f(s) & \text{if }|s| \leq\gamma, \\
            \phi(t)f^\eta(s)+\phi(1-t)f(s) & \text{if }|s| = \gamma(1 + t)\text{ for some } t\in[0,1], \\
            f_{\eta}(s) & \text{if }|s|\geq 2\gamma,
        \end{cases}        
    \end{equation}
    where $\phi(t):=t^2/(t^2+(1-t)^2)$ is a $C^1$ transition function ranging from $0$ to $1$. It is easily seen that $\tilde{f}$ is a $C^1$ function. In addition, for $s\in B_{2\gamma}\setminus B_{\gamma}$, we have:
    \begin{equation}\nonumber
        \tilde{f}(s) = f(s) + \phi\bigg(\frac{|s|-\gamma}{\gamma}\bigg)(f^\eta(s)-f(s)) 
    \end{equation}
    Thus;
    \begin{equation}\nonumber
    \begin{split}
        |\nabla \tilde{f}(s) - \nabla f(s)| & \leq ||\nabla \phi||_\infty|f^\eta(s)-f(s)|+||\phi||_\infty|\nabla f^\eta(s)-\nabla f(s)| \\ & \leq 2\omega_\eta(f) + \omega_\eta (\nabla f)<M/2.
    \end{split}
    \end{equation}
    This in turn implies that, for $s\in B_{2\gamma}\setminus B_{\gamma}$:
    \begin{equation}\nonumber
        |\nabla \tilde{f}(s)| \geq \bigg| |\nabla f(s)| - |\nabla \tilde{f}(s)-\nabla f(s)|\bigg| > M/2 > 0.
    \end{equation}
    Thus $\tilde{f}$ has non-zero gradient outside $B_\gamma$.
    
    Now, note that $\partial B_{3\gamma}$ is a $C^\infty$ manifold that is properly contained within $B_{\epsilon}\setminus B_{2\gamma}$ and $\tilde{f}:=f^\eta$ on $B_{\epsilon}\setminus B_{2\gamma}$ and thus $\tilde{f}^{-1}(c)\cap \text{Int}(B_{\epsilon}\setminus B_{2\gamma})$ is a $C^\infty$ manifold. If $\partial B_{3\gamma}$ intersects $\tilde{f}^{-1}(c)$ transversally, then setting $\delta=3\gamma$ yields $(i)$. Otherwise, applying tranversality theorem (c.f. Theorem \ref{thm:thoms_transversal} and the discussion in Appendix \ref{app:transversal}) we can find $\delta$ arbitrarily close to $3\gamma$ such that $B_{\delta}$ transversally intersects $\tilde{f}^{-1}(c)$, thus giving $(i)$.
    
    Finally, define $\tilde{f}_n$ analogously to $\tilde{f}$. It is immediate from the construction that $(ii)$ is satisfied. By Lemma \ref{lemma:mfld_to_R}, we can assume that $n$ is large enough that $f_n$ has non-zero gradient outside of $B_\gamma$. Further we can assume $n$ is large enough that $||\nabla f- \nabla f_n||_\infty < M/4$. Through similar reasoning to the above, it can be shown that $||\nabla f^\eta- \nabla f^\eta_{n}||_\infty < M/4$. It now follows that for $s\in B_{\epsilon}\setminus B_{2\gamma}$, we have:
    \begin{equation}\nonumber
    \begin{split}
            |\nabla \tilde{f}_n(s)| & \geq \bigg| |\nabla \tilde{f}(s)| -  |\nabla \tilde{f}_n(s)-\nabla \tilde{f}(s)|\bigg| \\
            & \geq M/2 - M/4 = M/4 > 0.
    \end{split}
    \end{equation}
    It now follows that condition $(iii)$ holds, completing the proof.
 \end{proof}

Before reading the proof of the following theorem, we recommend the reader be familiar with the statement of the mountain pass theorem given in Appendix \ref{app:mpt} Theorem \ref{thm:mpt} alongside its illustration, Fig. \ref{fig:MPT}.

\begin{mdframed}
    
    \begin{thm}\label{thm:not_quite_morse}
    Suppose that $S$ is a compact $C^1$ manifold with boundary, $f_n\xrightarrow{C^1}f$ where $f$ has isolated critical points and Assumptions \ref{assump:no_bdry_crits} and \ref{assumption:crit} hold. Then, as $n\rightarrow\infty$:
    \begin{equation}\nonumber
        N_m(f_n)\rightarrow N_m(f), \quad N_M(f_n)\rightarrow N_M(f), \quad \text{and}\quad  N_s(f_n)\rightarrow N_s(f).
    \end{equation}
    If $f$ has no points of undulation, then we additionally have that $N_C(f_n)\rightarrow N_C(f)$. Otherwise, $\limsup_{n\rightarrow\infty}N_C(f_n)\leq N_C(f)$.
    \end{thm}
    \vspace{0.3cm}
\end{mdframed}
\begin{proof}
    To begin, note that by definition:
    \begin{equation}\nonumber
        N_C(f_n) - N_0^H(f_n) = N_M(f_n) + N_m(f_n) + N_S(f_n) = \sum_{\lambda \neq 0} N_\lambda^H(f_n),
    \end{equation}
    and similarly for $f$. Thus, if $f$ has no undulation points (i.e. $N_0^H(f)=0$), then Theorem \ref{thm:homological_index} trivially implies that $N_C(f_n)\rightarrow N_C(f)$. Furthermore, we also have that:
    \begin{equation}\nonumber
        N_M(f_n) + N_m(f_n) + N_S(f_n)\rightarrow N_M(f) + N_m(f) + N_S(f).
    \end{equation}
    Suppose we were to show that $N_M(f_n)\rightarrow N_M(f)$. Then, by symmetry it easily follows that $N_m(f_n)\rightarrow N_m(f)$. Moreover, the above implies that $N_S(f_n)\rightarrow N_S(f)$ also. Thus it suffices to show that $N_M(f_n)\rightarrow N_M(f)$. To do so, noting Lemma \ref{lemma:mfld_to_R}, without loss of generality we now assume that $f$ has a single isolated critical point, $p$. We shall begin by considering the case that $p$ has non-zero homological index. Now, let $\epsilon$ be as in the proof of Theorem \ref{thm:homological_index}, so that eventually, $B:=B_{4\epsilon}(p)$ contains a single critical point of $f_n$, which we denote as $p_n$, and there are no critical points of $f_n$ outside $B$. We must show that $p$ is a local maxima if and only if $p_n$ is a local maxima for all but finitely many $n$.
    
    \underline{Step 1: if $p$ is a local maxima then, eventually, so is $p_n$.}  Suppose $p$ is a local maxima in $B$, but $p_n$ isn't a local maxima for infinitely many $n$. Without loss of generality, assume $p_n$ is not a local maxima for all $n$. For any $\delta >0$, we can apply Theorem \ref{thm:homological_index} to the manifold with boundary $B_{\delta}(p)$, to see that eventually $p_n\in B_{\delta}(p)$. It therefore follows that $p_n\rightarrow p$. By the uniform convergence $||f_n-f||_\infty\rightarrow 0$, it thus follows that $f_n(p_n)\rightarrow f(p)$.
    
    For each $n$, let $p_n^*:=\arg\max_{s\in B}f_n(s)$. Note that it is easily seen that $f_n(p_n^*)\rightarrow f(p)$. Now, assume, for contradiction, that $p_n^*\in \partial B$ for infinitely many $n$. Without loss of generality, assume that this holds for all $n$. By compactness, $\{p_n^*\}$ has a convergent subsequence which, for ease, we also denote as $\{p_n^*\}$, with limit $p^*$. Noting $p$ is a local maxima on $B$, choose $\delta>0$ to satisfy; 
    \begin{equation}\nonumber
        \delta < \frac{1}{2}\bigg(f(p) - \sup_{s\in \partial B} f(s)\bigg).
    \end{equation}
    Then, if we take $n$ large enough that $||f-f_n||_\infty<\delta$, we see that:
    \begin{equation}\nonumber
        f_n(p_n^*)\leq \sup_{s \in \partial B} f_n(s) \leq \sup_{s \in \partial B} f(s) + \delta < f(p) - \delta.
    \end{equation}
    This contradicts the fact that $f_n(p_n^*)\rightarrow f(p)$. Thus, it cannot be the case that $p_n^*\in \partial B$ infinitely often. It follows that $p_n^*\in\text{Int}(B)$ eventually and therefore $p_n^*$ must eventually be a local maxima of $f_n$. This means that $B$ contains two critical points of $f_n$, namely $p_n$ and $p_n^*$, infinitely often, contradicting the fact that $B$ eventually contains only a single critical point of $f_n$ by construction. Step 1 now follows. 

    \underline{Step 2: if $p_n$ is eventually a local maxima then so is $p$.} Suppose $p_n$ is eventually a local maxima of $f_n$ but $p$ is not a local maxima of $f$. Assume, without loss of generality, $p_n$ is a local maxima for all $n$. Note that, by assumption $p$ has non-zero homological index and thus cannot be a point of undulation. Suppose instead that $p$ is a local minima. By Step $1$, we have that if $p$ is a local minima then $p_n$ would also have to eventually be minima (this is seen by noting that local maxima of $-f$ and $-f_n$ are local minima of $f$ and $f_n$, and vice versa). This is a contradiction, as $p_n$ are eventually maxima, we know that $p$ cannot be a local minima. Thus, $p$ must be either a saddle point or local maxima.

   We now show that if $p_n$ is a sequence of local maxima, then $p$ cannot be a saddle point of $f$. It may seem intuitive that sequence of isolated local maxima cannot tend to a saddle point. However, care must be taken here. As Example \ref{example:isolated_crit} shows, without the condition that $p$ is an isolated critical point, counterexamples may easily be constructed. To prove the result, let $c=f(p)$ and $\epsilon > 0$ be as in the proof of Theorem \ref{thm:homological_index}, and suppose there exists $\delta \in (0,\epsilon)$ such that $f^{-1}(c)$ and $\partial B_\delta(p)$ intersect transversally (i.e. Condition \eqref{eq:condition} holds).

    Now, assume $n$ is large enough that $p_n\in B_\delta(p)$ is the only critical point of $f_n$ in $B_\delta(p)$. As $p$ is a saddle point, we can find $p'\in B_\delta(p)$ such that $f(p')>c$. Choose positive $\kappa < (f(p')-f(p))/2$. By uniform convergence, we have that $f_n(p_n)\rightarrow f(p)$. Thus, we can assume $n$ is large enough that $|f-f_n|<\kappa$ and $|f_n(p_n)-f(p)|<\kappa$. Therefore:
    \begin{equation}\nonumber
        f_n(p')\geq f(p')-\kappa > f(p)+\kappa \geq f_n(p_n).
    \end{equation}
    However, $p_n$ is a local maxima of $f_n$. Therefore, we can apply the mountain pass theorem for convex subsets, Theorem \ref{thm:mpt}, to see that there now exists a point $p_n' \in B_\delta(p)$ which is either a critical point or a point at which $f_n^{-1}(c)$ intersects $\partial B_\delta(p)$ tangentially. As $p_n$ is the only critical point of $f_n$ in $B_\delta(p)$ by construction, $f_n^{-1}(c)$ intersects $\partial B_\delta(p)$ tangentially at $p_n'$. 
    
    As $\partial B_\delta(p)$ is compact the sequence $\{p_n'\}$ must have a convergent subsequence tending to limit $p^*$. However, $c=f_n(p_n')\rightarrow f(p^*)$ and $0=\nabla f_n(p_n')\cdot (p_n'-p)\rightarrow \nabla f(p^*)\cdot (p^*-p)$. Thus, $f^{-1}(c)$ intersects $\partial B_\delta(p)$ tangentially at $p^*$. This is a contradiction as we assumed that $\delta>0$ was such that $f^{-1}(c)$ and $\partial B_\delta(p)$ intersect transversally. Thus, when such a $\delta >0$ exists, step 2 follows.

    Finally, suppose that for all $\delta>0$, $f^{-1}(c)$ and $\partial B_\delta(p)$ do not intersect transversally. If this is the case, then by Lemma \ref{lem:transversal}, we can construct transversal adjustments of $f$ and $f_n$, denoted $\tilde{f}$ and $\tilde{f}_n$ respectively, and choose $\delta>0$, such that $\tilde{f}^{-1}(c)$ and $\partial B_\delta(p)$ intersect transversally. Repeating the above argument for $\tilde{f}$ and $\tilde{f}_n$, we see that the claim of step 2 holds for $\tilde{f}$ and $\tilde{f}_n$. Noting that the critical points of $f$ and $\tilde{f}$ are identical, and that, for $n$ large enough, the critical points of $f_n$ and $\tilde{f}_n$ are identical, we see that step 2 now follows for $f$ and $f_n$.

    \underline{Step 3: Local maxima cannot converge to a point of undulation.} We now have that, under the assumption that $p$ has non-zero homological index, $p$ is a local maxima of $f$, if and only if, for $n$ large enough, $p_n$ is also a local maxima of $f_n$. We now must consider the case that the homological index of $p$ is zero, i.e. $p$ is a point of undulation. In this case, we must show that $f_n$ eventually possesses no local maxima.
    
    By Assumption \ref{assumption:crit}, we can define $B$ to be a neighbourhood of $p$ such that for $n$ large enough, $B$ contains at most one critical point of $f_n$. Using the arguments of Theorem \ref{thm:homological_index}, we may also assume $B^c$ contains no critical points of $f_n$. By choosing a subsequence of $\{f_n\}$ if needed, we now assume, without loss of generality, that $B$ contains exactly one critical point for all $n$, denoted $p_n$. We must show that it cannot be the case that $p_n$ is a local maxima infinitely often. Suppose for contradiction that $p_n$ is a local maxima infinitely often. Then $\text{Ind}^H_{\circ}(\nabla f)=0$ but, for infinitely many $n$, $\text{Ind}^H_{\circ}(\nabla f_n)\neq 0$. However, we can employ the same arguments used in the proof of Theorem \ref{thm:homological_index} (that is, we can construct a $C^1$ function which equals $f$ on a neighbourhood of $\partial B$ and $f_n$ in $B_{2\epsilon}(p)$ and then apply Poincar\'{e}-Hopf theorem) to obtain $\text{Ind}^H_{\circ}(\nabla f)=\text{Ind}^H_{\circ}(\nabla f_n)$ for sufficiently large $n$ (c.f. Equation \eqref{eq:hom_index_equality}). This is a contradiction. Thus, $B$ cannot contain a local maxima of $f_n$ infinitely often. It follows that for $n$ large enough, $f_n$ possesses no local maxima.
    
    Combining the above, we see that if $f$ has a single isolated critical point, then that critical point is a local maxima of $f$, if and only if, $f_n$ eventually has exactly one local maxima. Noting Lemma \ref{lemma:mfld_to_R}, we obtain $N_M(f_n)\rightarrow N_M(f)$, as desired. The final statement in the theorem then follows directly from Theorem \ref{thm:homological_index}.
    \end{proof}

\subsection{Hessian Convergence and Morse Theory}\label{sec:morse}

We now turn attention to the setting where $f_n\xrightarrow{C^2}f$ and consider what Morse theory can say about the relationship between the critical points of $f_n$ and $f$.
In this section, we make heavy reference to the proof of the Morse lemma given by \cite{ioffe1997}. We utilize this specific proof over the more standardly cited approaches, such as those of \cite{milnor1963morse} and \cite{nirenberg1974topics}, predominantly for two reasons. First, the standard proofs of the Morse lemma typically involve an application of the inverse and implicit function theorems, respectively, which both implicitly require $C^3$ differentiability or above. Secondly, unlike many of these proofs, the version of the Morse lemma given by \cite{ioffe1997} can be adapted to give explicit bounds on the size of the ``Morse neighborhood'' over which the polynomial representation holds (see Corollary \ref{corr:morse_constants} in Appendix \ref{app:morse}). Conventional approaches only guarantee that such a neighborhood exists, but do not quantify its size. Such quantification is crucial to our purposes as we need to describe the resemblance between the Morse neighborhoods of $f_n$ and $f$.

To begin, we show that, when $n$ is large enough, for every critical point of $f_n$, there exists at least one corresponding critical point of $f$. This statement can be shown using the methods of the previous section, and is formalized below.

\begin{mdframed}
    \begin{lemma}\label{lemma:seq}
        Suppose that $S$ is a compact $C^2$ manifold with boundary. Let $f_n:S\rightarrow \mathbb{R}$ be $C^1$, $f:S\rightarrow \mathbb{R}$ be $C^2$ and $p$ be a Morse point of $f$. If $f_n\xrightarrow{C^1}f$  and Assumption \ref{assump:no_bdry_crits} holds, then there exists a convergent sequence $p_n\rightarrow p$ such that each $p_n$ is a critical point of $f_n$, for all but finitely many $n$. If each $f_n$ is $C^2$ and $f_n\xrightarrow{C^2}f$, then for sufficiently large $n$, $\{p_n\}$ are Morse points with $\text{Ind}^M(\nabla f_n, p_n)= \text{Ind}^M(\nabla f, p)$. 
    \end{lemma}
    \vspace{0.2cm}
\end{mdframed}
\begin{proof}
    Noting Lemma \ref{lemma:mfld_to_R} assume $S$ is a compact subset of $\mathbb{R}^D$ containing a single critical point of $f$, $p\in \text{Int}(S)$. As $f$ is Morse, $p$ cannot be an undulation point and thus has non-zero homological index. Repeating the argument in the proof of Theorem  \ref{thm:homological_index}, we have that for any sufficiently small $\eta >0$, if $n$ is large enough then:
    \begin{equation}\nonumber
        0 \neq \text{Ind}^H(\nabla f, p) = \text{Ind}_\circ^H(\nabla f_n),
    \end{equation}
    where the interior homological index is defined on $B_\eta(p)$. Thus, for $n$ large enough, the interior homological index of $f_n$ on $B_\eta(p)$ is non-zero. Consequently, $f_n$ has a critical point $p_n$ in $B_\eta(p)$. As this holds for arbitrarily small $\eta>0$, we can now construct a sequence of points $p_n\rightarrow p$ such that $\nabla f_n(p_n)=0$ for all but finitely many $n$, as desired. \\
    \\
    To prove the second part of the statement, let each $f_n$ be $C^2$, with $\lambda_j(x)$ and $\lambda_{j,n}(x)$ denoting the $j^{th}$ eigenvalues\nomenclature{$\lambda_j,\lambda_{j,n}$}{Eigenvalues of $H_f$ and $H_{f_n}$, viewed as functions of space} of $H_f(x)$ and $H_{f_n}(x)$, respectively. By uniform convergence of the Hessians, we have that $|\lambda_{j,n}-\lambda_j|\rightarrow 0$ uniformly and thus $\lambda_{j,n}(p_n)\rightarrow \lambda_j(p)$. Therefore, for $n$ large enough $\text{sgn}(\lambda_{j}(p))=\text{sgn}(\lambda_{j,n}(p_n))$ for all $j$. The result now follows.
\end{proof}

It is worth noting that, even if we assume that $f_n$ is Morse and it's Hessian converges to that of $f$ pointwise, we still do not have convergence of the $N_\lambda(f_n)$ to $N_\lambda(f)$ for $\lambda\in\{1,...,D\}$. This is illustrated by the following example.

\begin{examplebox}
\begin{example}
    Suppose $f_n:S\rightarrow \mathbb{R}$ and $f:S\rightarrow \mathbb{R}$ are Morse, and $f_n\xrightarrow{C^1}f$ and $H_{f_n}(x) \rightarrow H_f(x)$ for all $x$, pointwise. Then, for $\lambda\in\{0,...,D\}$, it is still not necessarily true that $N^M_\lambda(f_n)\rightarrow N^M_\lambda(f)$.\\
    
{\centering
    \includegraphics[width=\textwidth]{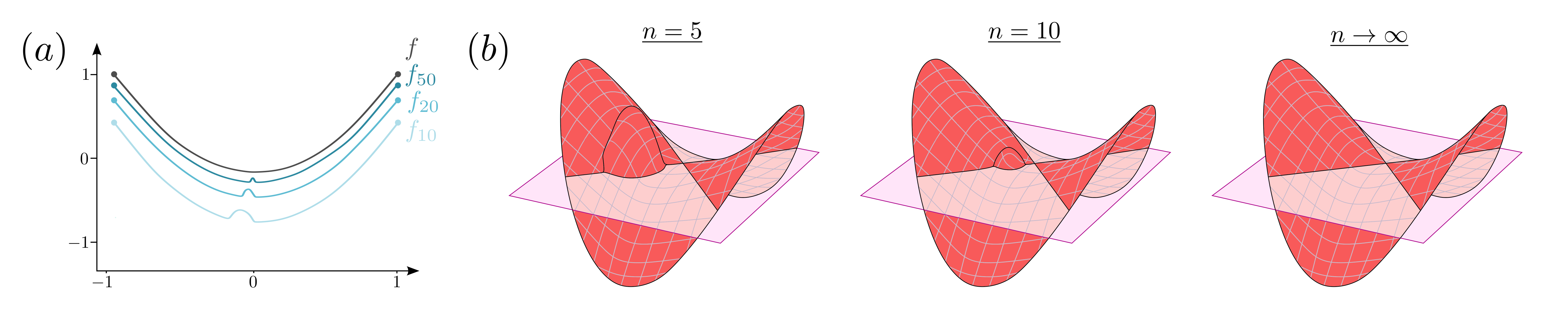}
    \captionof{figure}{Sequences of one-dimensional (a) and two-dimensional (b) functions $f_n\xrightarrow{C^1}f$ whose Hessians converge pointwise, but $N^M_\lambda(f_n)\not\rightarrow N^M_\lambda(f)$ as each $f_n$ has possesses a local maxima, but $f$ does not.}
    \label{fig:saddlewithbump}}
\raggedright
\vspace{0.2cm}
    The functions $f_n,f:\mathbb{R}\rightarrow \mathbb{R}$ in Panel $(a)$ of Fig \ref{fig:saddlewithbump} are given by:
    \begin{equation}\nonumber
        f_n(x)=x^2+b(nx+1)/\sqrt{n}-\frac{5}{n}\quad \text{ and }\quad f(x) = x^2,
    \end{equation}
    while, in a slight abuse of notation, the higher dimensional analogues $f_n,f:\mathbb{R}^2\rightarrow \mathbb{R}$ shown in Panel $(b)$ are given by: 
    \begin{equation}\nonumber
        f_n(x,y)=x^2-y^2+20b(nx+1,ny+1)/n^2\quad \text{ and }\quad f(x,y) = x^2 - y^2,
    \end{equation}
    where the smooth bump function $b:\mathbb{R}^m\rightarrow \mathbb{R}$ is defined as in Example \ref{example:singlemax}.
    
\end{example}
\vspace{0.2cm}
\end{examplebox}
The examples given in Fig. \ref{fig:saddlewithbump} demonstrate that the combination of $C^1$ convergence of $f_n$, Morseness of $f_n$ and $f$, and pointwise convergence of the Hessians is not sufficient to guarantee that the number of critical points converge. Thus, the stronger condition of $f_n\xrightarrow{C^2} f$ is indeed required. Before we turn our attention to Morse theory, it is worth building some intuition on how $C^2$ convergence may rule out situations such as those of Fig. \ref{fig:saddlewithbump} (a) and (b). 

In the one-dimensional case, it is not to difficult to show that $f_n\xrightarrow{C^2}f$, alongside Morseness of $f$, rules out situations such as (a). A sketch proof would be as follows; we know the local minima of $f_n$ in $(a)$ have positive second derivative and the local maxima have negative second derivative. As $n\rightarrow \infty$ these maxima and minima move closer together, and the uniform convergence of the Hessian forces the limiting critical point to thus have second derivative zero. This means that the minima of $f$ is degenerate, which cannot happen by the Morseness of $f$.

However, the same strategy cannot be applied to higher dimensional examples, such as $(b)$. To see why, first note that in $(b)$ the local maxima lies on the line $y=-x$, and moves increasingly closer to the saddle at the origin as $n\rightarrow \infty$. Applying the same argument as above, we see that along $y=-x$, $f$ must have a zero-valued directional second derivative. However, this is not a contradiction for higher dimensional critical points, as the saddle $f(x,y)=x^2-y^2$ is constant along the diagonal. Thus, to prove the general, multidimensional case, we require the machinery of Morse theory.

Specifically, we shall use the homotopic proof of the Morse lemma provided by \cite{ioffe1997}, a version of which is included in Appendix \ref{app:morse} for reference. As we draw heavily from this proof, it is worth briefly highlighting the key concepts underlying it. Informally, the Morse lemma states that for a Morse function $f:\mathbb{R}^D\rightarrow \mathbb{R}$ with a single isolated critical point at the origin, there exists a change of coordinates $\Gamma:\mathbb{R}^D\rightarrow \mathbb{R}^D$ such that locally $f(\Gamma(x))=x'Hx$\nomenclature{$\Gamma, \Gamma_n$}{Homeomorphisms guaranteed by the Morse lemma, locally mapping $f$ and $f_n$ to quadratic functions}. In other words, critical points appear approximately `polynomial' under a suitable choice of coordinate basis.

To prove that such a change of coordinates exist, \cite{ioffe1997} define $\phi(x):=f(x)-x'Hx$ and consider the homotopy $S_t:(f-\phi)\simeq f$ defined by $S_t(x)=x'Hx+t\phi(x)$. In words, $S_t$ may be thought of as a continuous deformation which morphs the surface $S_0(x)=x'Hx$ into $S_1(x)=f(x)$ (Fig. \ref{fig:morse_homotopy})\nomenclature{$S_t$}{Homotopy from $x'Hx$ to $f(x)$}. To find a diffeomorphism $\Gamma$ such that $S_0(x)=S_1(\Gamma(x))$, the key insight is to consider a second homotopy $\Gamma_t:\text{Id}\simeq \Gamma$. Suppose $\Gamma_t$ were constructed to ensure that $S_t \circ \Gamma_t$ were constant over all $t\in [0,1]$. Then, it would trivially follow that $x'Hx=(S_0 \circ \Gamma_0)(x)=(S_1 \circ \Gamma_1)(x)=f(\Gamma(x))$ as desired.\nomenclature{$\Gamma_t$}{Homotopy defined such that $S_t \circ \Gamma_t$ is constant for all $t$, not to be confused with $\Gamma_n$. $\Gamma_{t,n}$ is defined analogously}

{\centering
    \includegraphics[width=\textwidth]{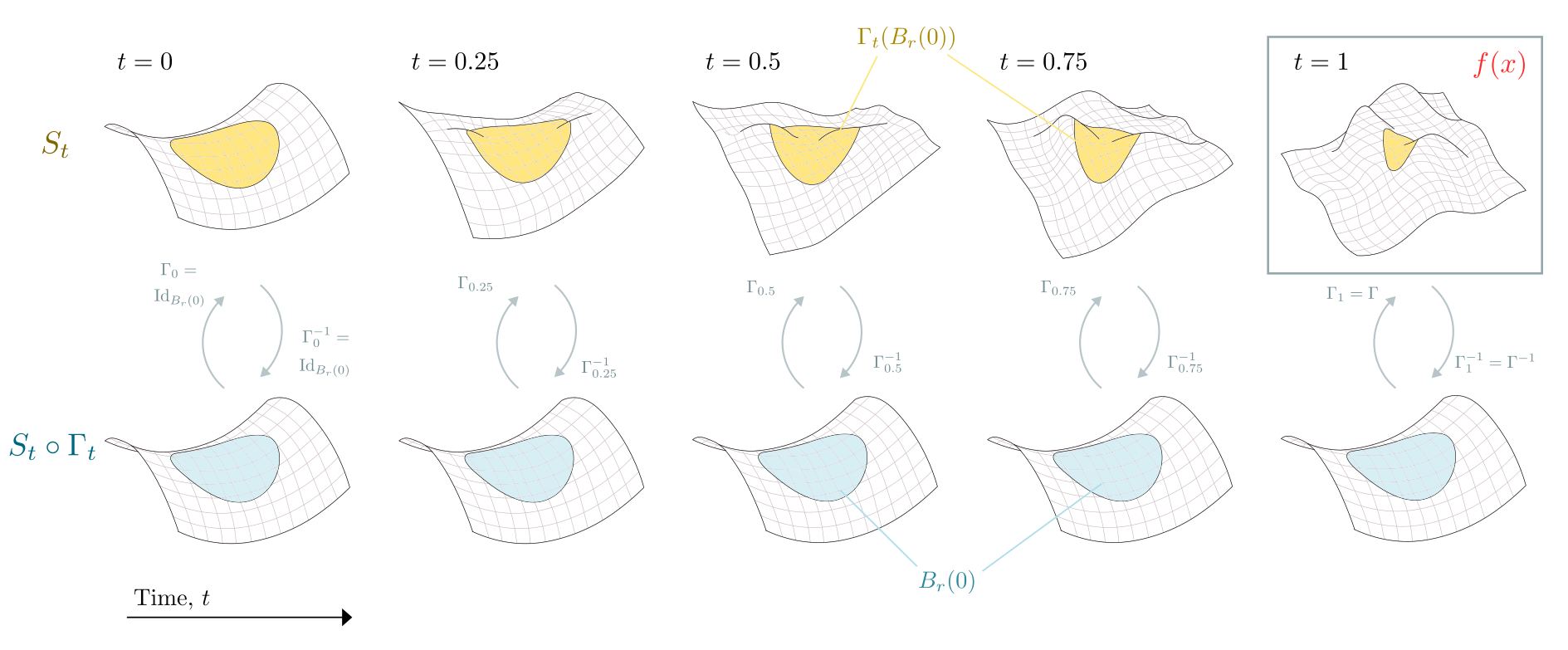}
    \captionof{figure}{Illustration of homotopies $S_t:(f-\phi)\simeq f$ and $\Gamma_t: \text{Id}\simeq \Gamma$ used by \cite{ioffe1997} to prove the Morse lemma. Here, $S_t$ is a continuous deformation from a `standard' saddle, $x^2-y^2$, into a saddle of $f$ (top row). The homotopy $\Gamma_t$ is constructed to ensure that $S_t \circ \Gamma_t$ is constant over time (bottom row). As $\Gamma_0$ is the identity map, it follows that $S_0 \circ \Gamma_0$ is the standard saddle (far left), and thus, as $\Gamma_t$ is constant over time, the map $\Gamma_1$ is a change of coordinates which transforms the saddle of $f$ to the standard saddle (far right). The region over which $\Gamma_t$ restricts to a diffeomorphism is highlighted in yellow for $S_t$ (top) and blue for $S_t\circ \Gamma_t$ (bottom).}
    \label{fig:morse_homotopy}}
    
    \vspace{0.5cm}
To construct $\Gamma_t$ in a way that guarantees that $S_t\circ \Gamma_t$ is constant, \cite{ioffe1997} consider the differential equation $\frac{\partial}{\partial t}(S_t \circ \Gamma_t)(x)=0$. By rearranging, they obtain $\Gamma_t$ as a non-autonomous flow of the form $\frac{\partial}{\partial t} \Gamma_t(x)=v_t(x)$ for some bi-Lipschitz function $v_t$\nomenclature{$v_t$}{Nonautonomous flow used to define $\Gamma_t$}. Along with the initial condition that $\Gamma_0(x)=\text{Id}(x)=x$, this flow expression can be used to implicitly define $\Gamma_t$, thus showing that such a homotopy does indeed exist. In the following proof, the given definitions of $\Gamma_t$ and $\Gamma_{t,n}$, which may appear to arise without clear motivation, were obtained using exactly this process.

\begin{mdframed}
    \begin{thm}
        \label{thm:constants_lemma}
         Suppose that $S$ is a compact $C^2$ manifold with boundary. Let $f_n:S\rightarrow \mathbb{R}$ be $C^2$ and $f:S \rightarrow \mathbb{R}$ be Morse. Assume that $\{p_n\}$ is a sequence of Morse points of $f_n$ tending to a Morse point of $f$, $p$. If $f_n \xrightarrow{C^2}f$ then there exists a convergent sequence of real positive values $r_n$ with limit $r>0$, a constant $\tilde{r}>0$, and bi-Lipschitz diffeomorphisms $\Gamma_n:B_{r_n}(p_n)\rightarrow \mathbb{R}^D$ and $\Gamma:B_{r}(p)\rightarrow \mathbb{R}^D$ satisfying:
        \begin{equation}\nonumber
            f_n(\Gamma_n(x)) = \frac{1}{2}(x-p_n)'H_{f_n}(p_n)(x-p_n), \quad f(\Gamma(x)) = \frac{1}{2}(x-p)'H_{f}(p)(x-p), \quad 
        \end{equation}
        and $|\Gamma_n-\Gamma|$ is well-defined on $B_{\tilde{r}}(p)$ for all but finitely many $n$, with $|\Gamma_n-\Gamma|\rightarrow 0$ uniformly on $B_{\tilde{r}}(p)$.
    \end{thm}
    \vspace*{0.2cm}
\end{mdframed}

\begin{proof}
    For ease, identify $p$ with the origin and, noting Lemma \ref{lemma:mfld_to_R}, assume $S=B_\eta(p)$. By translating each $f_n$ so that $f_n(x)\mapsto f_n(x-p_n)$, we can assume that each $p_n$ is also identified with $0$. It is easily verified that the resultant $f_n$ still satisfy $f_n\xrightarrow{C^2}f$.\\
    \\
    We now define $\Gamma_n$ and $\Gamma$ as in the proof of the Morse lemma employed by \cite{ioffe1997}. Specifically, we proceed by first defining:
    \begin{equation}\nonumber
        \phi(x)=f(x)-\frac{1}{2}x'Hx \quad \text{ and } \quad \phi_n(x)=f_n(x)-\frac{1}{2}x'H_nx.
    \end{equation}
    Next, we define $y_t$ and $y_{t,n}$\nomenclature{$y_t,y_{t,n}$}{Shorthand functions used in defining $v_t$ and $v_{t,n}$} as follows:
    \begin{equation}\nonumber
        y_t(x) = Hx + t\nabla \phi(x) \quad \text{ and } \quad y_{t,n}(x) = H_nx + t\nabla \phi_n(x),
    \end{equation}
    and $v_t$ and $v_{t,n}$ as follows:
\begin{equation}\nonumber
    v_t(x) = \begin{cases}
        -\phi(x) \frac{y_t(x)}{|y_t(x)|^2}, & x \neq 0, \\
        0, & x =0,\\
    \end{cases}\quad \text{ and } \quad 
    v_{t,n}(x) = \begin{cases}
        -\phi_n(x) \frac{y_{t,n}(x)}{|y_{t,n}(x)|^2}, & x \neq 0, \\
        0, & x =0.\\
    \end{cases}
\end{equation}
We then define $\Gamma_{t,n}$ and $\Gamma_t$ as the flows generated by the non-autonomous systems:
\begin{equation}\nonumber
    \frac{\partial }{\partial t}\Gamma_t(x) = v_t(\Gamma_t(x)), \quad \Gamma_0(x)=x, \quad \text{ and } \quad \frac{\partial }{\partial t}\Gamma_{t,n}(x) = v_{t,n}(\Gamma_{t,n}(x)), \quad \Gamma_{0,n}(x)=x,
\end{equation}
and denote $\Gamma:=\Gamma_1$ and, in a slight abuse of notation, $\Gamma_{n}:=\Gamma_{1,n}$.\\
\\
\noindent
Now that we have defined $\Gamma$ and $\Gamma_n$, we shall define the constants $r$ and $r_n$\nomenclature{$r,r_n$}{Radius of Morse neighborhoods for $\Gamma$ and $\Gamma_n$, respectively}. To do so, we let $K^1_n$ and $K^2_n$ be defined by:
\begin{equation}\nonumber
    K^1_n = \min\bigg(\frac{1}{2||H_n^{-1}||},  \frac{\ln(2)}{24||H_n||||H_n^{-1}||^2}\bigg) \quad \text{ and }\quad K^2_n = \frac{1}{8||H_n||||H_n^{-1}||}.
\end{equation}
Note that as $H_n$ is the Hessian of a Morse point, the above quantities are well-defined and positive for all $n$. Defining $K^1$ and $K^2$ analogously to the above, with all subscript $n$'s removed, we see that $K^1$ and $K^2$ are also well-defined and positive, with $K^1_n\rightarrow K^1$ and $K^2_n\rightarrow K^2$ as $n\rightarrow \infty$. For each $n$ let $\tilde{r}_n \in (0,\eta/2]$ be the largest possible value satisfying:
\begin{equation}\label{eq:ineq_rn}
    \sup_{x \in B_{\tilde{r}_n}}||H_{f_n}(x)-H_n|| \leq \frac{K_n^1}{2} \quad \quad\text{ and }\quad \tilde{r}_n \leq \frac{K_n^2}{2}.
\end{equation}
Now, let $r=\liminf_{n\rightarrow\infty}\tilde{r}_n$ and suppose that $r=0$. If this were the case, then we could find a convergent subsequence $\{\tilde{r}_{n_m}\}$ tending to zero. Noting that $K_{n_m}^2\rightarrow K^2 > 0$ and $\tilde{r}_{n_m} \rightarrow 0$, assume, for ease that $m$ is sufficiently large that $2\tilde{r}_{n_m}<\frac{K_{n_m}^2}{2}$.\\
\\
Now, by the supremum condition, it follows that, for each $m$, we can find $x_{n_m}\in B_{2\tilde{r}_{n_m}}$ such that $||H_{f_{n_m}}(x_{n_m})-H_{n_m}||\geq \frac{K_{n_m}^1}{2}$. Thus we have that:
\begin{equation}\nonumber
    \frac{K^1}{2} = \lim_{m\rightarrow \infty} \frac{K_{n_m}^1}{2} \leq \lim_{m\rightarrow \infty}\big|\big|H_{f_{n_m}}(x_{n_m})-H_{n_m}\big|\big| \leq
\end{equation}
\begin{equation}\nonumber
     \lim_{m\rightarrow \infty}\bigg[\big|\big|H_{f_{n_m}}(x_{n_m})-H_{f}(x_{n_m})\big|\big|+\big|\big|H_{f}(x_{n_m})-H\big|\big|+\big|\big|H-H_{f_{n_m}}(0)\big|\big|\bigg]=0
\end{equation}

\noindent
This is a contradiction, as $K^1>0$. Thus, noting that $r\geq 0$ by construction, it follows that $r>0$. Finally, for each $n$, define $r_n=\min(r,\tilde{r}_n)$. It follows that $r_n\rightarrow r$ and satisfies:

\begin{equation}\nonumber
    \sup_{x \in B_{r_n}}||H_{f_n}(x)-H_n|| < K_n^1  \quad\text{ and }\quad r_n < K_n^2.
\end{equation}

\noindent
Further, applying limits to both sides of \eqref{eq:ineq_rn}, it can be seen that an identical statement to the above holds with all subscript $n$'s removed. It now follows from Corollary \ref{corr:morse_constants}, that each $\Gamma_n$ and $\Gamma$ are well-defined bi-Lipschitz diffeomorphism on $B_{r_n}$ and $B_{r}$, respectively. We now define $\tilde{r}=\min(\lim\inf_{n}r_n, r)$. Using an identical argument to the above, it can be shown that for sufficiently large $n$, $\tilde{r}>0$. It follows that $\Gamma_n$ and $\Gamma$ are well-defined on $B_{\tilde{r}}$ for sufficiently large $n$.\\
\\
    All that remains to be shown now is that $|\Gamma_n-\Gamma|$ converges to zero uniformly on $B_{\tilde{r}}$. It is immediate from the above definitions that $|\phi_n-\phi|, |\nabla \phi_n - \nabla \phi|$ and, for any $t\in[0,1]$, $|y_{t,n}-y_t|$ all converge to zero uniformly. However, showing that  $|v_{t,n}-v_t|\rightarrow 0$ uniformly requires more care, due to the division by $y_{t,n}$ and the fact that $y_{t,n}(0)=0$. \\
    \\
    To show that $|v_{t,n}-v_t|\rightarrow 0$ uniformly on $B_{\tilde{r}}$, we first note that by Lemma \ref{lem:param} we have: 
    \begin{equation}\nonumber
        |x|\leq c^{-1}|y_t(x)| \quad \text{ and } \quad |x|\leq c_n^{-1}|y_{t,n}(x)|. 
    \end{equation}
    As in the proof of Lemma \ref{lem:morse} assume, without loss of generality that $\phi,\phi_n$ satisfy a Lipschitz condition on $B_{\tilde{r}}$ with constant $L<||H^{-1}||^{-1}$ and define $c:=(||H^{-1}||^{-1}-L)$ and $c_n:=(||H_n^{-1}||^{-1}-L)$. Following identical logic to that of the proof of Lemma \ref{lem:morse} (see Equation \eqref{eq:phi_bdd}), we obtain that $|\phi_n(x)|\leq L|x|^2$ and thus:
    \begin{equation}\label{eq:phi_bdds}
        \phi_n(x) \leq \frac{L}{c^2}|y_t(x)|^2, \quad \phi_n(x) \leq \frac{L}{c_n^2}|y_{t,n}(x)|^2 \quad \text{ and }\quad \phi_n(x) \leq \frac{L}{cc_n}|y_{t}(x)|\cdot|y_{t,n}(x)|,
    \end{equation}
    and, in addition, we also have that:
    \begin{equation}\label{eq:phi_diff_bdd}
    \begin{split}
        |\phi(x)-\phi_n(x)| & \leq \int_{0}^1 |\nabla \phi(tx) - \nabla \phi_n(tx)|\cdot |x| dt \leq \epsilon_{n}|x|. \\
    \end{split}
    \end{equation}
    where the first inequality follows by expressing $\phi$ and $\phi_n$ as path-integrals, and the second follows by defining $\epsilon_{n}:=\sup_{x\in B_{\tilde{r}}}|\nabla\phi(x)-\nabla \phi_n(x)|$. Now, consider the following inequality, which holds by basic manipulations:
    \begin{equation}\nonumber
            \bigg|\phi(x)|y_{t,n}(x)|^2 y_t(x)-\phi_n(x)|y_t(x)|^2 y_{t,n}(x)\bigg| \leq 
    \end{equation}
    \begin{equation}\nonumber
            |y_{t,n}(x)|^2 \cdot|y_t(x)| \cdot|\phi(x)-\phi_n(x)| + |\phi_n(x)|\cdot |y_t(x)|^2 \cdot |y_t(x)-y_{t,n}(x)| +
    \end{equation}
    \begin{equation}\label{eq:big_ineq2}
            |\phi_n(x)|\cdot |y_t(x)| \cdot \bigg| |y_t(x)|^2-|y_{t,n}(x)|^2\bigg| 
    \end{equation}
    We shall now place bounds on each of the three terms following the inequality. Starting with the first, we have:
    \begin{equation}\nonumber
        |y_{t,n}(x)|^2 \cdot|y_t(x)| \cdot|\phi(x)-\phi_n(x)| \leq \epsilon_{n}|y_{t,n}(x)|^2 \cdot|y_t(x)| \cdot |x| \leq c^{-1}\epsilon_{n}|y_{t,n}(x)|^2 \cdot|y_t(x)|^2,
    \end{equation}
    where the first inequality follows from Equation \eqref{eq:phi_diff_bdd} and the second follows from that fact $|x|\leq c^{-1}|y_t(x)|$. The second term of \eqref{eq:big_ineq2} can be bounded as follows:
    \begin{equation}\nonumber
        |\phi_n(x)|\cdot |y_t(x)|^2 \cdot |y_t(x)-y_{t,n}(x)| \leq \delta_n\frac{L}{c_n^2}|y_{t,n}(x)|^2\cdot |y_t(x)|^2,
    \end{equation}
    which follows by Equation \eqref{eq:phi_bdds} and setting $\delta_n:= \sup_{x\in B_{\tilde{r}}}|y_t(x)-y_{t,n}(x)|$.\\
    \\
    By expanding the square in the third term of \eqref{eq:big_ineq2} we obtain:
    \begin{equation}\nonumber
        |\phi_n(x)|\cdot |y_t(x)| \cdot \bigg| |y_t(x)|^2-|y_{t,n}(x)|^2\bigg| \leq |\phi_n(x)|\cdot |y_t(x)| \cdot | y_t(x)-y_{t,n}(x)|\cdot (|y_t(x)|+|y_{t,n}(x)|) 
    \end{equation}
    Bounding $|y_t-y_{t,n}|$ as before and rearranging yields:
    \begin{equation}\nonumber
        \leq \delta_n |\phi_n(x)|\cdot |y_t(x)|^2 +\delta_n |\phi_n(x)|\cdot|y_t(x)|\cdot|y_{t,n}(x)|.
    \end{equation}
    And employing Equation \eqref{eq:phi_bdds} gives:
    \begin{equation}\nonumber
        \leq \delta_n \bigg(\frac{L}{c_n^2}+\frac{L}{cc_n}\bigg)|y_{t,n}(x)|^2\cdot |y_t(x)|^2.
    \end{equation}
    Combining the three previous bounds, we obtain that:
    \begin{equation}\nonumber
        \bigg|\phi(x)|y_{t,n}(x)|^2 y_t(x)-\phi_n(x)|y_t(x)|^2 y_{t,n}(x)\bigg| \leq \bigg[\frac{\epsilon_n}{c}+\delta_n\bigg(\frac{2L}{c_n^2}+\frac{L}{cc_n}\bigg)\bigg]|y_{t,n}(x)|^2\cdot |y_t(x)|^2
    \end{equation}
    By dividing through both sides by $|y_{t,n}(x)|^2\cdot |y_t(x)|^2$ and noting the definition of $v_t$, we obtain that:
    \begin{equation}\nonumber
        |v_t(x)-v_{t,n}(x)| \leq \frac{\epsilon_n}{c}+\delta_n\bigg(\frac{2L}{c_n^2}+\frac{L}{cc_n}\bigg)
    \end{equation}
    The right hand side of the above tends to zero, and does not depend upon $x$. It therefore follows that $|v_t-v_{t,n}|$ converges to $0$ uniformly on $B_{\tilde{r}}$. Making $\tilde{r}$ smaller if necessary, it now follows from Lemma \ref{lem:flow_conv} that $|\Gamma_n-\Gamma|\rightarrow 0$ uniformly on $B_{\tilde{r}}$ as desired.\\
    \\
    To obtain the result in the form given, we need to translate back $f_n(x)\mapsto f_n(x+p_n)$ and $f(x)\mapsto f(x+p)$. Applying the corresponding translations to $\Gamma_n$ and $\Gamma$, we see that $|\Gamma_n-\Gamma|\rightarrow 0$ uniformly on $B_{\tilde{r}}(p_n)\cap B_{\tilde{r}}(p)$. To complete the proof we need only note that $p_n\rightarrow p$ and thus we can choose $N\in\mathbb{N}$ and $\tilde{r}'>0$ such that $B_{\tilde{r}'}(p)\subseteq B_{\tilde{r}}(p_n)\cap B_{\tilde{r}}(p)$ for all $n\geq N$. Thus, $|\Gamma_n-\Gamma|$ is well defined on $B_{\tilde{r}'}(p)$ for all but finitely many $n$ and $|\Gamma_n-\Gamma|\rightarrow 0$ uniformly on $B_{\tilde{r}'}(p)$.
\end{proof}

\begin{mdframed}
    \begin{thm}[Convergence of Morse Points]\label{thm:morse}
         Suppose that $S$ is a compact $C^2$ manifold with boundary. Let $f_n:S\rightarrow \mathbb{R}$ be $C^2$, $f:S\rightarrow \mathbb{R}$ be Morse and suppose that Assumption \ref{assump:no_bdry_crits} holds. If $f_n\xrightarrow{C^2}f$, then, for each $\lambda$: 
        \begin{equation*}
            \lim_{n\rightarrow \infty} N^M_\lambda(f_n) = N^M_\lambda(f)
        \end{equation*}
    \end{thm}
    \vspace*{0.3mm}
\end{mdframed}
\begin{proof} Assume, for ease, that $f$ has a single Morse point $p$ with index $\lambda$ and $S=B_\eta(p)$. By Lemma \ref{lemma:seq}, we can find $p_n\rightarrow p$ such that, for $n$ large enough, each $p_n$ is a Morse point of $f_n$ with index $\lambda$. By Theorem \ref{thm:constants_lemma}, for $n$ large enough, we can find $\tilde{r}\in(0,\eta/2]$, and diffeomorphisms $\Gamma_n$ and $\Gamma$ defined on $B_{\tilde{r}}(p)$ such that $|\Gamma_n-\Gamma|\rightarrow 0$ and:
\begin{equation}\nonumber
    f_n(\Gamma_n(x)) = \frac{1}{2}(x-p_n)'H_{f_n}(p_n)(x-p_n).  
\end{equation}
Let $0<R<\tilde{r}$ and $M= \inf_{x \in S \setminus \Gamma(B_{R}(p))}|\nabla f(x)|$. Suppose $M=0$. It follows we can find $x \in \overline{S \setminus \Gamma(B_{R}(p))}$ such that $|\nabla f(x)|=0$. Thus, $x\neq p$ is a critical point of $f$. However, as $f$ is Morse and $p$ is the only Morse point of $f$, it follows that $x$ cannot be a critical point of $f$. This is a contradiction and thus $M>0$.\\
\\
Assume $n$ is large enough that $|\nabla f-\nabla f_n|< M/2$. It follows that for $x \in S \setminus\Gamma(B_{R}(p))$, we have $|\nabla f_n (x)| \geq \frac{M}{2} > 0$. It follows that, for suitably large $n$, $f_n$ has no Morse points outside $\Gamma(B_{R}(p))$, but \textit{at least one} Morse point inside $\Gamma(B_{R}(p))$. All that remains to be shown is that, for sufficiently large $n$, $f_n$ has \textit{exactly one} Morse point inside $\Gamma(B_{R}(p))$. \\
\\
To show this, suppose that $p_n\neq p_n'$ are Morse points of $f_n$ inside $\Gamma(B_{R}(p))$. By Lemma \ref{lem:images_nested}, it follows that, for $n$ large enough, $\Gamma(B_{R}(p))\subseteq \Gamma_n(B_{\tilde{r}}(p))$. Let $s_n = \Gamma_n^{-1}(p_n)$ and $s_n'= \Gamma_n^{-1}(p_n')$. By the bijectivity of $\Gamma_n$, we have that $s_n\neq s_n'$. However, we also have:
\begin{equation}\nonumber
    \nabla(f_n \circ \Gamma_n)(s_n) = (\nabla f_n)(\Gamma_n(s_n)) \cdot \nabla \Gamma_n (s_n)  = \nabla f_n(p_n) \cdot \nabla \Gamma_n (s_n) = 0.
\end{equation}
as $p_n$ is a Morse point of $f_n$ and thus $\nabla f_n(p_n)=0$. Similarly, we have that $\nabla(f_n \circ \Gamma_n)(s_n')=0$. Now, note that as $p_n$ is a Morse point of $f_n$, $H_{f_n}(p_n)$ is full rank. Thus:
\begin{equation}\nonumber
    \nabla(f_n \circ \Gamma_n)(x)=(x-p_n)'H_{f_n}(p_n)=0\quad \text{ if and only if }\quad x=p_n.
\end{equation}
Therefore $s_n=s_n'=p_n$. It now follows that there can only be one morse point of $f_n$ inside $\Gamma(B_{R}(p))$. Thus, for $n$ large enough $f_n$ has exactly one Morse point, with index $\lambda$. Noting Lemma \ref{lemma:mfld_to_R}, we can generalize to the case where $S$ is an arbitrary compact $C^2$ manifold with boundary to see that there is a one-to-one correspondence between the Morse points of index $\lambda$ of $f_n$ and those of $f$. The result now follows.
\end{proof}

\section{Probabilistic Results}\label{sec:prob}

In this section, we use the results of the previous to make statements about the convergence of random processes. Following the empirical processes framework of \cite{van1996weak}, we define a `random process' to be a random variable $G$\nomenclature{$G_n$}{Empirical random process derived from $n$ observations}\nomenclature{$G$}{Limiting random process, $G_n\rightsquigarrow G$} whose values are continuous scalar functions. Formally, given a probability space $(\Omega, \mathcal{F}, \mathbb{P})$\nomenclature{$(\Omega, \mathcal{F}, \mathbb{P})$}{Probability space}, a random process maps events $\omega \in \Omega \mapsto f \in \mathbf{F}$ where $\mathbf{F}\subseteq C(S,\mathbb{R})$\nomenclature{$\mathbf{F}$}{Arbitrary function space $\mathbf{F}\subseteq C(S,\mathbb{R})$}. Following standard convention, we will drop all $\omega$ in our notation. As before, $S$ is a compact manifold with boundary and, so that we may use the more convenient form of Assumption \ref{assumption:crit}, we now additionally assume that $S$ is also a metric space.

As the measurability of some of the events we consider is not guaranteed, we shall employ the notion of outer probability, defined for $B\subseteq\Omega$, by $\mathbb{P}^*[B]=\inf \{\mathbb{P}[A]: A \supseteq B, A \in \mathcal{F}\}$\nomenclature{$\mathbb{P},\mathbb{P}^*$}{Probability, outer probability}. As is convention in the setting of empirical processes, we will be interested in a sequence of random processes $\hat{G}_n$ converging weakly to $G$, denoted $\hat{G}_n\rightsquigarrow G$\nomenclature{$\rightsquigarrow$}{Weak convergence}. Following \cite{van1996weak}, we denote convergence in outer probability as $\xrightarrow{P*}$\nomenclature{$\xrightarrow{P*}$}{Convergence in outer probability} and almost uniform convergence as $\xrightarrow{a.u.}$\nomenclature{$\xrightarrow{a.u.}$}{Almost uniform convergence}. For completeness, a list of standard definitions and lemmas used in this section can be found in Supplementary Material section \ref{app:prob}.

For notational ease, we define the following random variables:
\begin{equation}\nonumber
    L := \sup_{s \in \partial S} |\nabla G(s)|, \quad R := \inf_{s_1,s_2\in Z(\nabla G)} |s_1-s_2| \quad \text{ and } \quad \hat{R}_n := \inf_{s_1,s_2\in Z(\nabla \hat{G}_n)} |s_1-s_2|.
\end{equation}\nomenclature{$L$}{Random variable representing the smallest observed gradient of $G$ on $\partial S$}\nomenclature{$R$}{Random variable representing the smallest distance between critical points of $G$}\nomenclature{$\hat{R}_n$}{Random variable representing the smallest distance between critical points of $\hat{G}_n$}
We are now in a position to state a probabilistic version of Theorems \ref{thm:homological_index} and \ref{thm:not_quite_morse}.
\begin{mdframed}
\begin{thm}[Probabilistic Homological Index Convergence]\label{thm:prob_hom}
     Let $S$ be a compact $C^1$ manifold with boundary and a metric space. Suppose $\{\hat{G}_n\}$ and $G$ are random processes which take their values in $C^1(S,\mathbb{R})$ and satisfy the following:
    \begin{enumerate}[label=(\subscript{H}{{\arabic*}})]
        \item $(\hat{G}_n, \nabla \hat{G}_n)\rightsquigarrow (G, \nabla G)$ where $(G, \nabla G)$ is separable,
        \item $\mathbb{P}^*[L=0]=0$ and $\mathbb{P}^*[R=0]=0$,
        \item There exists $\delta >0$ such that $\sum_{n=1}^\infty \mathbb{P}^{*}[\hat{R}_n < \delta] < \infty$,
        \item $N_M(G), N_m(G), N_S(G)$ and $\{N_\lambda^H(G)\}_{\lambda \neq 0}$ are Borel measurable.
    \end{enumerate}
    then the following convergences hold:
    \begin{equation}\label{eq:conv_res}
        N_M(\hat{G}_n)\rightsquigarrow N_M(G), \quad N_m(\hat{G}_n)\rightsquigarrow N_m(G), \quad N_S(\hat{G}_n)\rightsquigarrow N_S(G), \quad N^H_\lambda(\hat{G}_n)\rightsquigarrow N^H_\lambda(G)
    \end{equation}
    for $\lambda\neq 0$. If, in addition, the below assumption is satisfied:
    \begin{enumerate}[label=(\subscript{H}{{\arabic*}})]
        \item[($H_5$)] $\mathbb{P}^*[N^H_0(G) = 0]=1$,
    \end{enumerate}
    then $N_C(\hat{G}_n)\rightsquigarrow N_C(G)$.
    \vspace{0.2cm}
\end{thm}
\end{mdframed}
\begin{remark}
    Conditions like $(H_1)$ typically arise via central limit theorems and are not uncommon in the literature (c.f. \cite{adler1981geometry}, \cite{chung2020introductionrandomfields}). The condition $(H_2)$ essentially says critical points of $G$ are almost surely isolated and inside the interior of $S$; such assumptions can also be found in \cite{davenport2022confidenceregionslocationpeaks}. 
    
    Condition $(H_3)$ is a little subtler, but essentially says that the critical points of $\hat{G}_n$ are not likely to become $\delta$-close as $n\rightarrow \infty$. A common example in which such a condition may be assumed is when $\hat{G}_n$ has been derived by smoothly interpolating data that was recorded on a fixed evenly-spaced lattice. In such situations, the distance between critical points is typically bounded below by the resolution of the lattice, and thus $\mathbb{P}^{*}[\hat{R}_n < \delta] =0$ for all $n$ if $\delta$ is chosen to be sufficiently smaller than the distance between gridpoints. $(H_4)$ is a standard measurability constraint and $(H_5)$ assumes that the probability of seeing an undulation point in the limiting process is zero, which matches practical experience in most settings.
\end{remark}
\begin{proof}
    To begin, let $(\tilde{G}_n, \nabla \tilde{G}_n)\xrightarrow{a.u.} (\tilde{G}, \nabla \tilde{G})$ be almost sure representations of $(\hat{G}_n, \nabla \hat{G}_n)\rightsquigarrow (G, \nabla G)$ (c.f. Lemma \ref{thm:as_reps}), with $\tilde{L}$, $\tilde{R}$ and $\tilde{R}_n$ defined analogously. By $(H_3)$ and the Borel-Cantelli lemma, Lemma \ref{lem:borelcantelli}, we have that:
    \begin{equation}\nonumber
        \mathbb{P}^*\bigg[\limsup_{n\rightarrow\infty}\{\tilde{R}_n<\delta\}\bigg] = \mathbb{P}^*\bigg[\liminf_{n\rightarrow\infty}\tilde{R}_n<\delta \bigg] = 0.
    \end{equation}
    Now, let $N$ be any of the convergent functions in \eqref{eq:conv_res} and $\epsilon > 0$ be arbitrary. As $(\tilde{G}_n, \nabla \tilde{G}_n)\xrightarrow{a.u.} (\tilde{G}, \nabla \tilde{G})$, we can find measurable $A\subseteq \Omega$ such that $\mathbb{P}[A]\geq 1-\epsilon$ and $\tilde{G}_n\xrightarrow{C^1} \tilde{G}$  uniformly over $\omega \in A$. Thus, we have:
    \begin{equation}\nonumber
        \begin{split}
            \mathbb{P}[A]=\mathbb{P}^{*}[A] & \leq \mathbb{P}^{*}\bigg[A \land \bigg\{\big\{\tilde{L}=0\big\} \lor \big\{\tilde{R}=0\big\} \lor \big\{\liminf_{n\rightarrow\infty}\tilde{R}_n < \delta \big\}\bigg\}\bigg] \\
            &  + \mathbb{P}^{*}\bigg[A \land \big\{\tilde{L}>0\big\} \land \big\{\tilde{R}>0\big\} \land \big\{\liminf_{n\rightarrow\infty}\tilde{R}_n \geq \delta \big\}\bigg]
        \end{split}
    \end{equation}
    by subadditivity of the outer measure. However, again by subadditivity, we now have that the above is less than or equal to:
    \begin{equation}\nonumber
        \begin{split}
            & \leq \mathbb{P}^{*}[\tilde{L}=0] +\mathbb{P}^{*}[\tilde{R}=0] + \mathbb{P}^{*}\bigg[\liminf_{n\rightarrow\infty}\tilde{R}_n < \delta \bigg] \\
            &  + \mathbb{P}^{*}\bigg[A \land \big\{\tilde{L}>0\big\} \land \big\{\tilde{R}>0\big\} \land \big\{\liminf_{n\rightarrow\infty}\tilde{R}_n \geq \delta \big\}\bigg].
        \end{split}
    \end{equation}
    But $\{\tilde{L}=0\}$ and $\{\tilde{R}=0\}$ both have zero outer probability by $(H_2)$ and $\{\liminf_{n\rightarrow\infty}\tilde{R}_n \leq \delta \}$ has zero outer probability by the above, thus first three terms above are zero. Now consider $\omega$ in the intersection event of the fourth term. By the definition of $A$, we have that realisations of $\tilde{G}_n$ and $\tilde{G}$ converge in the $C^1$ metric. As they satisfy $\tilde{L}>0$, the boundary condition, Assumption  \ref{assump:no_bdry_crits}, is also satisfied, and considering the definitions of $\tilde{R}_n$ and $R$ we see that Assumption \ref{assumption:crit} holds and $G$ has isolated critical points. Thus, we see that for any $\omega$ in this event, the assumptions of Theorems \ref{thm:homological_index} and \ref{thm:not_quite_morse} are satisfied. Therefore:
    \begin{equation}\nonumber
        1-\epsilon \leq \mathbb{P}^{*}\bigg[A \land \big\{\tilde{L}>0\big\} \land \big\{\liminf_{n\rightarrow\infty}\tilde{R}_n \geq \delta \big\}\bigg] \leq \mathbb{P}^*\bigg[\lim_{n\rightarrow \infty}N(\tilde{G}_n)= N(\tilde{G})\bigg]
    \end{equation}
    Letting $\epsilon\rightarrow 0$, it follows that $N(\tilde{G}_n)\xrightarrow{P*}N(\tilde{G})$ and thus, noting that $\tilde{G}_n$ and $\tilde{G}$ are equal in law to $\hat{G}_n$ and $G$, $N(\hat{G}_n)\rightsquigarrow N(G)$, as desired (c.f. Lemma \ref{lem:out_implies_weak}). Noting the final line of Theorem \ref{thm:homological_index}, we see that when $(H_4)$ holds, the above also implies that $\mathbb{P}^*[\lim_{n\rightarrow\infty}N_0^H(\tilde{G}_n)=N_0^H(\tilde{G})]=1$, from which the convergence $N_C(\hat{G}_n)\rightsquigarrow N_C(G)$ can now be derived.
\end{proof}

A probabilistic variant of Theorem \ref{thm:morse} may also be given. To do so, however, we must first define the variable $M$ as follows:
\begin{equation}\nonumber
    M=\inf_{s \in S}\bigg(\max\big(|\nabla G(s)|, ||H_G(s)||\big)\bigg).
\end{equation}
Note that stating $M=0$ is equivalent to saying that $G$ has a degenerate critical point. \nomenclature{$M$}{Random variable defined such that $M=0$ if and only if $G$ is not Morse} Thus, if $M>0$ then $G$ is Morse.

\begin{mdframed}
\begin{thm}[Probabilistic Morse Index Convergence]\label{thm:prob_morse}
     Let $S$ be a compact $C^2$ manifold with boundary and a metric space. Suppose $\{\hat{G}_n\}$ and $G$ are random processes which take their values in $C^2(S,\mathbb{R})$ and satisfy the following:
    \begin{enumerate}[label=(\subscript{M}{{\arabic*}})]
        \item $(\hat{G}_n, \nabla \hat{G}_n, H_{\hat{G}_n})\rightsquigarrow (G, \nabla G, H_{G})$ where $(G, \nabla G, H_{G})$ is separable.
        \item $\mathbb{P}^*[L=0]$,
        \item $\mathbb{P}^*[M=0]$,
        \item $\{N_\lambda^M(G)\}_{\lambda \in \mathbb{N}}$ are Borel measurable.
    \end{enumerate}
    then the following convergences hold:
    \begin{equation}\label{eq:conv_res_morse}
        N_M(\hat{G}_n)\rightsquigarrow N_M(G), \quad N_m(\hat{G}_n)\rightsquigarrow N_m(G), \quad N_S(\hat{G}_n)\rightsquigarrow N_S(G) 
        \end{equation}
        \begin{equation}\nonumber
        N_C(\hat{G}_n) \rightsquigarrow N_C(G), \quad \text{ and } \quad N^M_\lambda(\hat{G}_n)\rightsquigarrow N^M_\lambda(G),
    \end{equation}
    for all $\lambda\in \{0,...,D\}$. 
    \vspace{0.2cm}
\end{thm}
\end{mdframed}
\begin{proof}
    Using identical logic to the proof of Theorem \ref{thm:prob_hom} and defining $\tilde{M}$ analogously, we see that $\hat{G}_n$ and $G$ have almost sure representations given by $\tilde{G}_n$ and $\tilde{G}$ such that:
        \begin{equation}\nonumber
        1-\epsilon \leq \mathbb{P}^{*}\bigg[A \land \big\{\tilde{L}>0\big\} \land \big\{\liminf_{n\rightarrow\infty}\tilde{R}_n \geq \delta \big\} \land \{\tilde{M}>0\}\bigg].
    \end{equation}
    Applying Theorem \ref{thm:morse}, noting that $M>0$ implies that $G$ is Morse, and taking the limit as $\epsilon \rightarrow 0$, we obtain:
    \begin{equation}\nonumber
        \mathbb{P}^*\bigg[\lim_{n\rightarrow\infty}N_\lambda^M(\tilde{G}_n)=N_\lambda^M(\tilde{G})\bigg]=1
    \end{equation}
    for any $\lambda \in \mathbb{N}$. Thus $N_\lambda^M(\tilde{G}_n)\rightsquigarrow N_\lambda^M(\tilde{G})$ and therefore $N_\lambda^M(\hat{G}_n)\rightsquigarrow N_\lambda^M(G)$. The remaining results follow by noting the relations between $N_M, N_m, N_S, N_C$ and $\{N_\lambda^M\}_{\lambda \in \mathbb{N}}$.
    
\end{proof}

\section{Discussion}\label{sec:discussion}

In this work, we have provided theorems on the convergence of the number critical points under various regularity assumptions. To aid practical use, our theorems consider a range of theoretical assumptions, and have also been expressed in the language of empirical processes in Section \ref{sec:prob}.

A natural question is whether this work could be extended to describe the convergence properties of other topological invariants, such as Betti numbers or the Euler characteristic. For instance, for many applications, it is desirable to estimate the Euler characteristic of an excursion set of the form $\mathcal{A}_c:=\{s\in X: f(s)\geq c\}$, for some known predefined threshold $c\in\mathbb{R}$. Defining $\hat{\mathcal{A}}_c:=\{s\in X: f_n(s)\geq c\}$, it might be asked if our proofs can be used to show that $\chi_n(\hat{\mathcal{A}}_c)\rightarrow\chi(\mathcal{A}_c)$ as $n\rightarrow \infty$. Such a question is well-founded as, under mild assumptions, the Euler characteristic of a $C^2$ manifold with boundary $S$ can be expressed as $\chi(S):=\sum_\lambda (-1)^\lambda N_C(f)$ for any Morse $f:S\rightarrow \mathbb{R}$. However, showing the desired convergence is not straightforward as, in this instance, not only is $f_n$ varying with $n$, but also the spatial domain $\hat{\mathcal{A}}_c$ itself. As we have more to say on this topic, we intend to address this question further in future work.

Another question of potential interest is whether the theorems hold under different assumptions placed on the domain $S$. To address this question, we claim, but do not expound specifics, that our proofs easily extend to cases in which $S$ is either a manifold with corners or a Whitney stratified manifold. We highlight this, as such settings are common in the literature, and may be found in texts such as \cite{nicolaescu2007invitation} and \cite{Adler2009} respectively.

\section*{Acknowledgements}

The authors would like to thank Prof.~Thomas E.~Nichols and Prof.~Armin Schwartzman for their thoughtful comments regarding early iterations of this work. TM would also like to express strong gratitude to Dr.~Sam Power for his insights on optimisation methods for bang-bang controllers, which informed the development of a preliminary version of the counterexample in Supplementary Section \ref{supp:cantor}, albeit one that was not included in the final work.
 
\newpage
\printnomenclature

\newpage

\bibliographystyle{unsrt}
\bibliography{References.bib}

\newpage
\appendix
\appendixpage
\addappheadtotoc
\section{Supporting Lemmas}\label{app:lemmas}

\begin{mdframed}
    \begin{lemma}\label{lem:lipschitz}
        Suppose $\phi:S\rightarrow \mathbb{R}$ is $C^2$ with $H_\phi(0)=0$. Then, for all $L>0$, there exists an $R_0>0$ satisfying $L\geq \sup_{x\in B_{R_0}}||H_\phi(x)||$, such that $\nabla \phi$ is Lipschitz on $B_{R_0}$ with constant $L$.
    \end{lemma}
\vspace{0.25cm}
\end{mdframed}
\begin{proof}
Fix $L>0$ and assume $S:=B_\eta$ (c.f Lemma \ref{lemma:mfld_to_R}). As $H_\phi(0)=0$ and $H_\phi$ is continuous, we can choose $R_0\in(0,\eta]$ such that $||H_\phi(x)||\leq L$ for $x\in B_{R_0}$. By mean value theorem, for $x,y\in B_{R_0}$:
\begin{equation}\nonumber
    |\nabla\phi(x)-\nabla\phi(y)| \leq \bigg|\int_0^1 H_{\phi}\big(x+t(x-y)\big)dt \cdot |x-y|\bigg| \leq L|x-y|
\end{equation}
as desired.
\end{proof}

\begin{mdframed}
    \begin{lemma}\label{lem:flow_conv}
        Let $R>0$ and suppose $\{v_{t,n}\}_{n\in\mathbb{N}}$ is a sequence of functions defined on $B_R\subseteq S$, uniformly converging to $v_t$. Assume further that $v_t$ is Lipschitz in $B_R$ with constant $L$, not dependent on $t$. If $\Gamma_t(x)$ and $\Gamma_{t,n}(x)$ are flows generated by the non-autonomous systems:
\begin{equation}\nonumber
    \frac{\partial }{\partial t}\Gamma_t(x) = v_t(\Gamma_t(x)), \quad \Gamma_0(x)=x, \quad \text{ and } \quad \frac{\partial }{\partial t}\Gamma_{t,n}(x) = v_{t,n}(\Gamma_{t,n}(x)), \quad \Gamma_{0,n}(x)=x,
\end{equation}
that are all well-defined on some $B_r$ for some $r>0$, then we have that $|\Gamma_{t,n}-\Gamma_t|\rightarrow 0$ uniformly on $B_{r'}$ for all $t\in [0,1]$ and some $0<r'<r$.
    \end{lemma}
    \vspace{0.2cm}
\end{mdframed}
\begin{proof}
    Using the FTC we have that, for all $t\in[0,1]$:
    \begin{equation}\nonumber
        \Gamma_t(x) = x + \int_{s=0}^{t} v_t(\Gamma_t(x)) ds, \quad \text{ and }\quad\Gamma_{t,n}(x) = x + \int_{s=0}^{t} v_{t,n}(\Gamma_{t,n}(x)) ds.
    \end{equation}
    Therefore;
    \begin{equation}\nonumber
        |\Gamma_t(x) -\Gamma_{t,n}(x)| = \bigg|\int_{s=0}^{t} v_{t}(\Gamma_{t}(x)) - v_{t,n}(\Gamma_{t,n}(x)) ds\bigg| \leq
    \end{equation}
    \begin{equation}\nonumber
    \int_{s=0}^{t} |v_{t}(\Gamma_{t}(x)) - v_{t}(\Gamma_{t,n}(x))| ds+\int_{s=0}^{t} |v_{t}(\Gamma_{t,n}(x)) - v_{t,n}(\Gamma_{t,n}(x))| ds.
    \end{equation}
    By the uniform convergence of $v_{t,n}$ to $v_t$, the second term can be bounded by $t\epsilon_n<\epsilon_n$ where $\epsilon_n:=||v_t-v_{t,n}||\rightarrow 0$. Assume that $0<r'<r$ is small enough that $\Gamma_t(B_{r'})\subseteq B_{R}(0)$ for all $t\in [0,1]$, and assume $x \in B_{r'}$. Noting the Lipschitz condition on $v_t$, we now have:
    \begin{equation}\nonumber
        |\Gamma_t(x) -\Gamma_{t,n}(x)| \leq \epsilon_n +
    L \int_{s=0}^{t} |\Gamma_{t}(x) - \Gamma_{t,n}(x)| ds
    \end{equation}
    Applying Gr\"{o}wall's inequality, we now get that:
    \begin{equation}\nonumber
        |\Gamma_t(x) -\Gamma_{t,n}(x)| \leq \epsilon_n e^{Lt} \leq \epsilon_n e^{L}
    \end{equation}
    Noting that the right-hand side of the above does not depend on $x$ and tends to zero as $n\rightarrow \infty$ completes the proof.
\end{proof}

\begin{mdframed}
    \begin{lemma}\label{lem:images_nested}
        Let $0<r_1<r_2$, $p\in V$ and $S$ and $V$ be compact metric spaces with $B_{r_2}(p)\subset V$. Suppose, $f_n,f:V\rightarrow S$ are homeomorphisms with $f_n\rightarrow f$ uniformly on $V$. Then, for $n$ large enough, $f(B_{r_1}(p)) \subseteq f_n(B_{r_2}(p))$. 
    \end{lemma}
    \vspace{0.2cm}
\end{mdframed}
\begin{proof}
    For brevity denote $B_i:=B_{r_i}(p)$ for $i=1,2$. To begin, define $\kappa_n$ as:
    \begin{equation}\nonumber
        \kappa_n:=\inf\big\{|s_1-s_2|: s_i \in \partial f_n(B_i) \text{ for } i \in \{1,2\}\big\}
    \end{equation}
    Note that as $\partial f_n(B_1)\times \partial f_n(B_2)$ is closed, the above infinimum is attained at some $(s_1^n, s_2^n)$. For each $n$, choose such an $s_1^n\in\partial f_n(B_1)$ and $s_2^n\in\partial f_n(B_2)$ satisfying $\kappa_n=|s_1^n-s_2^n|$. Noting that each $f_n$ is a homeomorphism and $\overline{B_2}\subseteq V$ as $V$ is closed, we can also find $t_1^n\in \partial B_1$ and $t_2^n\in \partial B_2$ such that $s_1^n=f_n(t_1^n)$ and $s_2^n=f_n(t_2^n)$. \\
    \\
    Now suppose it were the case that $\liminf_{n\rightarrow\infty}\kappa_n = 0$. Then we would be able to find a subsequence $\{n_m\}$ such that:
    \begin{equation}\nonumber
        \lim_{m\rightarrow \infty}|s_1^{n_m}-s_2^{n_m}| = \lim_{m\rightarrow \infty}|f_{n_m}(t_1^{n_m})-f_{n_m}(t_2^{n_m})| = 0
    \end{equation}
    By compactness $\{t_1^{n_m}\}$ and $\{t_2^{n_m}\}$ have convergent subsequences. Without loss of generality, denote these subsequences again as $\{t_1^{n_m}\}$ and $\{t_2^{n_m}\}$ with limits given by $t_1^*$ and $t_2^*$, respectively. Note that by construction $t_1^*\in \partial B_1$ and $t_2^*\in \partial B_2$, so $t_1^*\neq t_2^*$. By the uniform convergence of $f_n$ to $f$, we see that for each $i\in\{1,2\}$:
    \begin{equation}\nonumber
        |f_{n_m}(t_i^{n_m})-f(t_i^*)| \leq ||f_{n_m}-f||_\infty + |f(t_i^{n_m})-f(t_i^*)|\xrightarrow{m\rightarrow\infty} 0
    \end{equation}
    Combining this with the above we see that:
    \begin{equation}\nonumber
        f(t_2^*)-f(t_1^*)  \leq \lim_{m\rightarrow \infty}|f(t_2^*)-f_{n_m}(t_2^{n_m})|+|f_{n_m}(t_1^{n_m})-f(t_1^*)|=0
    \end{equation}
    Thus, $f(t_2^*)=f(t_1^*)$. As $f$ is a homeomorphism we have that $t_2^*=t_1^*$. However, this is a contradiction, as $t_2^*\neq t_1^*$. Thus, $\liminf_{n\rightarrow\infty}\kappa_n = \kappa$ for some $\kappa>0$. It now follows from the definition of $\kappa$ that, for $n$ large enough, $(f_n(B_1))_{\frac{\kappa}{2}}\subset f_n(B_2)$.\\
    \\
    Finally assume that $n$ is large enough that $||f_n-f||_\infty<\frac{\kappa}{2}$. If $y\in f(B_1)$, it follows that there exists an $x\in B_1$ such that $f(x)=y$ and thus $|f_n(x)-f(x)|<\frac{\kappa}{2}$. Therefore, there exists $y'=f_n(x)\in f_n(B_1)$ such that $|y-y'|<\frac{\kappa}{2}$. As $y$ was arbitrary it follows that $f(B_1) \subset (f_n(B_1))_{\frac{\kappa}{2}}$. Combining with the above, yields  $f(B_1) \subset f_n(B_2)$, as desired.
\end{proof}

\section{The Poincar\'{e}-Hopf Theorem}\label{app:ph}

In this appendix, we state the generalized Poincar\'{e}-Hopf theorem, as given by \cite{jubin2009generalized}. For completeness, a definition of the notion of `orientability' employed in the statement of the theorem is given in Supplemental Material Section \ref{supp:orient}. For the purposes of the main text, it is sufficient to note that the closed $\eta$-ball in $\mathbb{R}^D$ is always trivially oriented and thus the below theorem can be applied in the settings of the main text.

\begin{mdframed}
    \begin{thm}[The Generalized Poincar\'{e}-Hopf Theorem]\label{thm:poincarehopf} Let $v$ be a vector field on a compact oriented manifold with boundary $S$. Then:
    \begin{equation}\nonumber
        \text{Ind}^H(v) := \begin{cases}
            \chi(S) & \text{if }D \text{ is even},\\
            0 & \text{if }D \text{ is odd,}\\
        \end{cases}
    \end{equation}
    where $\chi(S)$ is the Euler Characteristic of $S$ and $\text{Ind}^H$ is the homological index given by Definition \ref{def:homological_index}.
    \end{thm}
    \vspace{0.2cm}
\end{mdframed}
\begin{proof}
    See Theorem 12 of \cite{jubin2009generalized}.
\end{proof}

\section{Thom's Transversality Theorem}\label{app:transversal}

In this appendix, we state the transversality theorem cited in the proof of Lemma \ref{lem:transversal}, alongside necessary definitions. To do so, we first define the concept of `transversal intersection of smooth manifolds'.

\begin{mdframed}
\begin{defn}[Transversal Intersection of Smooth Manifolds]\label{def:tang_travs}
    Suppose $M$ and $N$ are  $C^\infty$ sub-manifolds of a $C^\infty$ manifold $Y$. We say that $M$ and $N$ intersect transversally if, for all $x \in M \cap N$,
\begin{equation}\nonumber
T_x M+T_x N=T_x Y.
\end{equation}
If $M$ and $N$ intersect transversally, we write $M \pitchfork N$. If $M$ and $N$ do not intersect transversally, we say they intersect \textit{tangentially}. Alternatively, we may say $M$ and $N$ are \textit{tangential} at the point $x$ for which the above condition is violated.
\end{defn} 
\end{mdframed}
This definition can be extended to speak of the transversal intersection of a smooth function with a smooth manifold. To state the definition, we first denote by $D_a F$ the differential of a function $F:X\rightarrow Y$ at a point $x \in X$, which is a linear map $D_x F : T_x X \to T_{F(x)} Y$.

\begin{mdframed}
\begin{defn}[Transversal Intersection of a Smooth Function and Smooth Manifold]\label{def:transversal2}
    Let $X,Y$ and $Z$ be $C^\infty$ manifolds, with $Z\subseteq Y$ and $f: X \rightarrow Y$ be a smooth map. We say that $f$ is transverse to $Z$ if, for every $x \in f^{-1}(Z)$,
\begin{equation}\nonumber 
D_x f\left(T_x X\right)+T_{f(x)} Z=T_{f(x)} Y,
\end{equation}
where $+$ here represents summation in the sense of vector spaces. If $f$ is transversal to $Z$, we write $f \pitchfork Z$.
\end{defn} 
\end{mdframed}
The below statement of the transversality theorem is taken from \cite{guillemin2010differential}. A gentle introduction to the topic can be also be found in \cite{Greenblatt2015AnIT}.
\begin{mdframed}
    \begin{thm}[Thom's Transversality Theorem]\label{thm:thoms_transversal} Let $X, Y, Z$ and $R$ be $C^\infty$ manifolds with $Z\subseteq Y$. Suppose that $G: X \times R \rightarrow Y$ is a smooth map and define $g_r:X\rightarrow Y$ by $g_r(x):=G(x,r)$. If $G\pitchfork Z$, then for almost all $r \in R$, $g_r \pitchfork Z$.
    \end{thm}
    \vspace{0.2cm}
\end{mdframed}
In the proof of Lemma \ref{lem:transversal}, we are working on the $C^\infty$ manifold $Y:=\text{Int}(B_\epsilon\setminus B_{\epsilon/2})$ and wish to show that the $C^\infty$ manifolds $Z:=\tilde{f}^{-1}(c)\cap Y$ and $\partial B_{\delta}$ intersect transversally for some $\delta$ arbitrarily close to $3\epsilon/4$. To show this formally, noting that $Y\approx \mathbb{S}^{n-1}\times (\epsilon/2,\epsilon)$, we first apply a change of coordinates to represent each $y\in Y$ as $y=(x_1,...,x_{n-1},r)$ where $(x_1,...,x_{n-1})\in X:=\mathbb{S}^{n-1}$ and $r\in R:=(\epsilon/2,3\epsilon/2)$ is the distance from $y$ to the origin. 

Define $G:X\times R \rightarrow Y$ by $G((x_1,...,x_{n-1}),r):=(x_1,...,x_{n-1},r)$. As $G$ is in effect the identity map, we can see trivially that, for all $w \in X\times R$, and thus for all $w \in G^{-1}(Z)$, we have $D_w G\left(T_w (X \times R)\right)=\mathbb{R}^n=T_{G(w)} Y$. By definition \ref{def:transversal2}, we now have $G\pitchfork Z$. Let $g_r:X\rightarrow Y$ be defined as $g_r(x):=G(x,r)$. Applying Thom's Transversality Theorem, Theorem \ref{thm:thoms_transversal}, we can choose $\delta\in R$ be such that $g_\delta \pitchfork Z$ and $\delta$ is arbitrarily close to $3\epsilon/4$. Unwinding the definitions, we see that:
\begin{equation}\nonumber
\begin{aligned}
g_\delta \pitchfork Z 
&\iff D_x g_\delta(T_x \mathbb{S}^{n-1}) + T_{g_\delta(x)} Z = T_{g_\delta(x)} Y \quad \text{for all } x \in g_\delta^{-1}(Z)
\quad && \\
&\iff T_y(\partial B_\delta) + T_y(\tilde{f}^{-1}(c) \cap Y) = T_y Y 
\quad  \text{for all } y \in \partial B_\delta \cap \tilde{f}^{-1}(c)  && \\
&\iff \partial B_\delta \pitchfork \left( \tilde{f}^{-1}(c) \cap Y \right),
\quad && 
\end{aligned}
\end{equation}
and thus the claim holds.

\section{The Mountain Pass Theorem}\label{app:mpt}

The following variant of the mountain pass theorem is adapted from Theorem 17.6 of \cite{Jabri_2003}, which is in turn attributed to \cite{convex_mpt}. It is employed in the proof of Theorem \ref{thm:not_quite_morse} in the main text. An illustration of the theory, alongside a description to aid intuition is provided by Fig. \ref{fig:MPT}. 

\begin{mdframed}
    \begin{thm}[Mountain Pass Theorem on Convex Domains]\label{thm:mpt} Let $S\subseteq\mathbb{R}^D$ be compact and convex and $f:S\rightarrow \mathbb{R}$ be $C^1$. Suppose further that $p_1, p_2 \in S$ satisfy $f(p_1) \leq f(p_2)$ and, for every path $\gamma:[0,1]\rightarrow S$ from $p_1$ to $p_2$, $\min_{t \in [0,1]} f(\gamma(t)) < f(p_1)$. Then, there exists a $p_3 \in S$ such that
\begin{equation}\nonumber
    f(p_3):= c < f(p_1)\quad  \text{ and }\quad g(p_3) := \sup_{\substack{q \in S\\ |p_3-q|<1}} \nabla f(p_3)\cdot (q-p_3) = 0 .
\end{equation}
In particular, if $\partial S$ is a $C^1$ manifold then either $p_3$ is a critical point of $f$, or $f^{-1}(c)$ is tangential to $\partial S$ at $p_3$.
    \end{thm}
\end{mdframed}

{\centering
    \includegraphics[width=\textwidth]{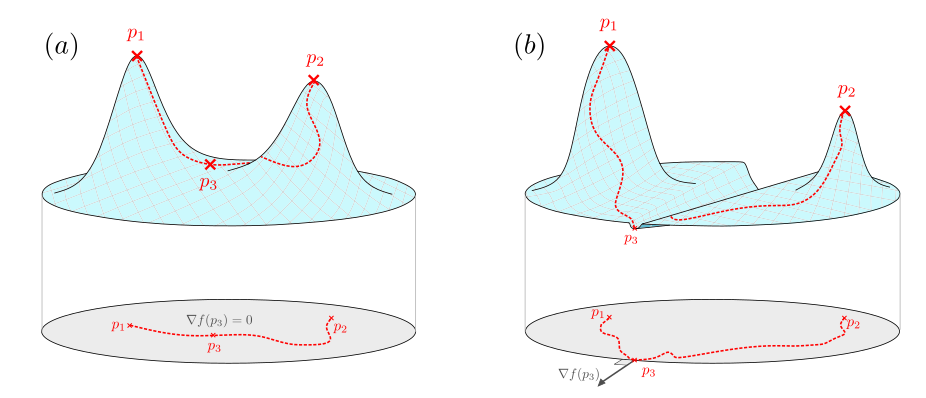}
    \captionof{figure}{Illustration of the mountain pass theorem for convex domains. The theorem states that, given two local maxima, $p_1$ and $p_1$, on a convex domain with smooth boundary, it must the case that either there exists a third critical point (Panel (a)), or there exists a point on the boundary of the domain at which the gradient is normal to the boundary (Panel (b)). In either case, the third point, $p_3$, is typically found by identifying a path from $p_1$ to $p_2$ whose lowest point of elevation is as high as possible.}
    \label{fig:MPT}}
    \begin{proof} See Section 17.2 of \cite{Jabri_2003}. In the notation of \cite{Jabri_2003}, we have that $M:=S, K:=[0,1], K^{*}:=\{0,1\}, \Phi:=-f, \Gamma$ is the set of paths from $p_1$ to $p_2$ and $\gamma^*$ is an arbitrary fixed choice of path from $p_1$ to $p_2$. The fact that the Palais-Smale condition on $M$ is satisfied, and that the limit of the sequence $\{u_n\}$ in \cite{Jabri_2003} exists, both follow from the compactness of $S$. \\
    \\
    The final statement in the theorem follows as if $p_3\in \text{Int}(S)$, then we can choose $\lambda>0$ sufficiently small that $q':=p_3 + \lambda\nabla f(p_3)\in \text{Int}(S)$. Noting the definition of $g$, we see that:
    \begin{equation}\nonumber
        g(p_3) \geq\nabla f(p_3)\cdot(p_3 + \lambda\nabla f(p_3)-p_3)=\lambda \nabla f(p_3) \cdot \nabla f(p_3) = \lambda|\nabla f(p_3)|^2.
    \end{equation}
    Thus, $g(p_3)=0$ implies $\nabla f(p_3)=0$. If instead $p_3\in \partial S$, for any unit tangent vector $w^*$ of $\partial S$ at $p_3$, we can choose a sequence of points $q_n\in S$ such that $q_n-p_3/|q_n-p_3|$ tends to $w^*$. As $g(p_3)=0$, we have that, for all $n$, $0 \geq \nabla f(p_3)\cdot (q_n-p_3)/|q_n-p_3|\rightarrow \nabla f(p_3)\cdot w^*$. As this holds for all such tangent vectors, including $-w^*$, we see that $\nabla f(p_3)\cdot w^*\leq 0$ and $-\nabla f(p_3)\cdot w^*\leq 0$, and thus $\nabla f(p_3)\cdot w^*= 0$. It therefore follows that $\nabla f$ is normal to $\partial S$ at $p_3$. The result now follows.
    \end{proof}

\section{The Morse Lemma}\label{app:morse}

In this appendix, we recap the proof of the Morse Lemma given by \cite{ioffe1997} for $C^{1,1}$ functions ($C^1$-functions with locally Lipschitz gradients). As our interest lies only in $C^{2}$ functions, for concision the below proofs have been restricted to the $C^2$ setting. 

\begin{mdframed}
\begin{lemma}[The Morse Lemma]\label{lem:morse} Suppose that $S$ is a $C^{2}$ manifold and let $f:S\rightarrow \mathbb{R}$ be Morse. Assume $f$ has a Morse point at the origin and, as in the main text, denote $H:=H_f(0)$. Then, there exists $r >0$ depending only upon $||H||$ and $||H^{-1}||$ and a zero-preserving $C^1$ bi-Lipschitz homeomorphism $\Gamma(x)$ defined on the ball $B_r$ such that:
\begin{equation}\nonumber
   f(\Gamma(x)) = \frac{1}{2}x'Hx.
\end{equation}
\end{lemma}
\vspace*{0.1cm}
\end{mdframed}
In the main text, the above is used to show Theorem \ref{thm:constants_lemma}. As our interest lies in the specific neighborhood upon which the homeomorphism, $\Gamma$, is defined the following write-up is given with the specific aim of carefully tracking the constants used to define $r$. Beyond aesthetic differences, however, the proof closely follows that of \cite{ioffe1997} and is included purely for reference.\\
\\
Before we prove the lemma, it will be useful to define $\phi:S \rightarrow \mathbb{R}$ as follows:
\begin{equation}\nonumber
    \phi(x)=f(x)-\frac{1}{2}x'Hx.
\end{equation}
Further, we shall let $R_0>0$ be such that $\nabla \phi$ satisfies a Lipschitz condition on $B_{R_0}$ for some constant $L\leq ||H^{-1}||^{-1}$. The existence of such an $R_0>0$ follows from Lemma \ref{lem:lipschitz} and the fact that $\phi(0)=0$ and $\nabla \phi(0)=0$, by the definition of $f$. Note that by Lemma \ref{lem:lipschitz}, we can choose $L$ as small as we like, by adjusting the value of $R_0$.\\
\\
Now to prove Lemma \ref{lem:morse}, we must first show the following.
\begin{mdframed}
\begin{lemma}\label{lem:param} Define $H$, $R_0$ and $L$ as above. For $t \in [0,1]$, define $y_t$ as follows:
    \begin{equation}\nonumber
        y_t(x) = Hx + t\nabla \phi(x),
    \end{equation}
    and let $c=(||H^{-1}||^{-1}-L)$, $C=(||H||+L)$ and $R=cR_0/C$. It follows that $y_t|_{B_R}$ is a homeomorphism and for any $x, x' \in B_R$, we have that:
    \begin{equation}\nonumber
        c |x - x' | \leq | y_t(x) - y_t(x') | \leq C |x-x'|
    \end{equation} 
\end{lemma}
\vspace*{0.05cm}
\end{mdframed}
\begin{proof}
    First consider the equation $v=y_t(x)$. By rearranging we have that:
    \begin{equation}\nonumber
        x = H^{-1}v - tH^{-1}\nabla\phi(x) := z_v(x)
    \end{equation}
    Denote the expression on the right as $z_v(x)$. Fix $v \in B_{cR_0}$. We shall now show that $z_v$ maps $B_{R_0}$ into itself. Let $x \in B_{R_0}$. It now follows that:
    \begin{equation}\nonumber
        \begin{split}
            |z_v(x)| & \leq  ||H^{-1}||\cdot(|v| + |\nabla\phi(x)|) \\
            & \leq  ||H^{-1}||\cdot(cR_0 + L|x|) \\
            & \leq  ||H^{-1}||R_0\cdot(c + L) \\
            & = R_0 \\
        \end{split}
    \end{equation}
    where the first inequality follows from the definition of $z_v$ and the fact $t\in[0,1]$, the second follows from the definition of $v$ and the Lipschitz condition on $\nabla \phi$, the third from the fact $x \in B_{R_0}$. The final inequality is due to the definition of $c$.\\
    \\
    Now let $x,x'\in B_{R_0}$. We have that:
    \begin{equation}\nonumber
        \begin{split}
            |z_v(x)-z_v(x')| & = |H^{-1}v-tH^{-1}\nabla \phi(x)-H^{-1}v+tH^{-1}\nabla \phi(x')| \\
            & \leq t||H^{-1}|| \cdot |\nabla \phi(x)-\nabla\phi(x')| \\
            & \leq tL ||H^{-1}|| \cdot |x-x'|
        \end{split}
    \end{equation}
    where the last inequality follows from the Lipschitz condition on $\nabla\phi$. As $t\in[0,1]$ and $L<||H^{-1}||^{-1}$, $z_v:B_{R_0}\rightarrow B_{R_0}$ is a contraction mapping. Applying Banach fixed point theorem, $z_v$ must have a unique fixed point in $B_{R_0}$ (i.e. $z_v(x)=x$). Returning to the definition of $z_v$, it follows there is a unique $x\in B_{R_0}$ such that $y_t(x)=v$, whenever $v \in B_{cR_0}$.\\
    \\
    We now show that if $x\in B_{R}$ then $y_t(x)\in B_{cR_0}$. This follows as:
    \begin{equation}\nonumber
    \begin{split}
        |y_t(x)| & \leq ||H|| \cdot |x| + |\nabla \phi(x)-\nabla \phi(0)| \\
         & \leq (||H||+L)R \\
         & \leq \frac{CcR_0}{C}= cR_0\\
    \end{split}
    \end{equation}
    where the first inequality follows from the fact $\nabla \phi(0)=0$ and $t\in[0,1]$, the second inequality follows from the Lipschitz condition on $\nabla \phi$ and the fact $x \in B_R$ and the remaining inequalities follow from the definitions of $C$ and $R$. Combining the above, it now follows that $y_t$ restricts to a homeomorphism on $B_R$, as desired.\\
    \\
    We now show that $y_t$ is lower Lipschitz with constant $c$. This follows as, for any $x,x' \in B_R$;
    \begin{equation}
    \begin{split}
        |y_t(x)-y_t(x')| & \geq |H(x-x')| - t|\nabla \phi(x)-\nabla \phi(x')| \\
        & \geq  ||H^{-1}||^{-1}|x-x'| - |\nabla \phi(x)-\nabla \phi(x')| \\
        & \geq  (||H^{-1}||^{-1}-L)|x-x'| \\
        & = c|x-x'|
    \end{split}
    \end{equation}
    where the first inequality follows by the definition of $y_t$, the second by the fact that for any $x$, we have $|x| = |H^{-1}Hx| \leq ||H^{-1}||\cdot |Hx|$, and the third inequality follows from the Lipschitz condition on $\nabla \phi$.\\
    \\
    To show the upper Lipschitz condition, note that:
    \begin{equation}\nonumber
    \begin{split}
        |y_t(x)-y_t(x')| & \leq ||H|| \cdot |x-x'| + |\nabla \phi (x)-\nabla \phi (x')|\\
        & \leq (||H|| + L) \cdot |x-x'| \\
        & = C |x-x'| \\
    \end{split}
    \end{equation}
    where the inequality follows from the Lipschitz condition on $\nabla \phi$. This completes the proof.
\end{proof}
We are now in a position to prove Lemma \ref{lem:morse}.

\begin{proof}[Proof of Lemma \ref{lem:morse}]
    We begin by noting that, by the Lipschitz condition on $\nabla \phi$, for any $x,x'\in B_{R_0}$:
\begin{equation}\nonumber
\begin{split}
    \int_0^1 |\nabla \phi(tx)- \nabla\phi(tx')| dt & \leq \int_0^1 tL|x-x'| dt = L |x-x'|
\end{split}
\end{equation}
We now have that:
\begin{equation}\nonumber
    \begin{split}
        |\phi(x)-\phi(x')| & = \bigg|\int_0^1 \nabla \phi(tx) \cdot x dt- \int_0^1 \nabla \phi(tx') \cdot x' dt\bigg| \\
        & \leq \int_0^1 |\nabla \phi(tx)| \cdot |x-x'| -\big|\nabla \phi(tx')-\nabla \phi(tx)\big| \cdot |x'| dt \\
        & \leq L (|x|+|x'|)\cdot |x-x'|
    \end{split}
\end{equation}
where the first equality follows by considering the integral of $\frac{d}{dt}\phi(P(t))$ along the path $P(t)=tx, t \in [0,1]$ and the final inequality follows from the Lipschitz condition on $\nabla \phi$.\\
\\
Define $y_t$ as in Lemma \ref{lem:param}. It follows from Lemma \ref{lem:param} that for any $x \in B_R$, $|x|\leq c^{-1}|y_t(x)|$. Thus, for $x,x'\in B_R$:
\begin{equation}\label{eq:phi_bdd}
\begin{split}
    |\phi(x)-\phi(x')| & \leq \frac{L}{c}(|y_t(x)|+|y_t(x')|) \cdot|x-x'| \\
   & \leq \frac{L}{c^2}(|y_t(x)|+|y_t(x')|)\cdot|y_t(x)-y_t(x')| \\
\end{split}
\end{equation}
Substituting $x'=0$, it can also be seen that for any $x \in B_R$:
\begin{equation}\label{eq:v_cts}
    |\phi(x)|\leq L|x|^2 \leq \frac{L}{c^2}|y_t(x)|^2
\end{equation}
We now define the function $v_t$ on $B_R$ as follows:
\begin{equation}\nonumber
    v_t(x) = \begin{cases}
        -\phi(x) \frac{y_t(x)}{|y_t(x)|^2}, & x \neq 0 \\
        0 & x =0\\
    \end{cases}
\end{equation}
Noting that $v_t$ is continuous, we shall show that $v_t$ is Lipschitz with constants depending only on $||H||$ and $||H^{-1}||$. Let $x,x'\in B_R$ and assume, for ease, that $|y_t(x)|\geq |y_t(x')|$. By rearranging and applying the triangle inequality, we obtain:
\begin{equation}\nonumber
        \bigg|\phi(x)|y_t(x')|^2 y_t(x)-\phi(x')|y_t(x)|^2 y_t(x')\bigg| \leq 
\end{equation}
\begin{equation}\nonumber
        |y_t(x')|^2 \cdot|y_t(x)| \cdot|\phi(x)-\phi(x')| + |\phi(x')|\cdot |y_t(x)|^2 \cdot |y_t(x)-y_t(x')| +
\end{equation}
\begin{equation}\label{eq:big_ineq}
        |\phi(x')|\cdot |y_t(x)| \cdot (|y_t(x)|^2-|y_t(x')|^2) 
\end{equation}
Considering the three terms in turn, we see that by applying Equation \eqref{eq:phi_bdd};
\begin{equation}\label{eq:term1}
    |y_t(x')|^2 \cdot|y_t(x)| \cdot|\phi(x)-\phi(x')| \leq \frac{L}{c}|y_t(x')|^2 \cdot|y_t(x)| \cdot\bigg(|y_t(x)|+|y_t(x')|\bigg)\cdot |x-x'|
\end{equation}
For the second term, we have that, by applying Equation \eqref{eq:phi_bdd} and the Lipschitz condition for $y_t$ given in Lemma \ref{lem:param}:
\begin{equation}\label{eq:term2}
    |\phi(x')|\cdot |y_t(x)|^2 \cdot |y_t(x)-y_t(x')| \leq \frac{L}{c}\cdot |y_t(x')|^2\cdot |y_t(x)|^2 \cdot C|x-x'|
\end{equation}
For the third term, we have that:
\begin{equation}\nonumber
        |\phi(x')|\cdot |y_t(x)| \cdot (|y_t(x)|^2-|y_t(x')|^2) \leq 
\end{equation}
\begin{equation}\nonumber
         |\phi(x')|\cdot |y_t(x)| \cdot \bigg(|y_t(x)| + |y_t(x')|\bigg) \cdot |y_t(x)-y_t(x')|  \leq 
\end{equation}
\begin{equation}\label{eq:term3}
         \frac{L}{c^2}|y_t(x')|^2\cdot |y_t(x)| \cdot \bigg(|y_t(x)| + |y_t(x')|\bigg) \cdot C|x-x'|
\end{equation}
where the first inequality follows from expanding the product of squares, noting that $|y_t(x)|\geq |y_t(x')|$ by assumption and applying the reverse triangle inequality, whilst the second inequality follows from again applying Equation \eqref{eq:phi_bdd} and the Lipschitz condition for $y_t$ given in Lemma \ref{lem:param}.\\
\\
Combining the inequalities in \eqref{eq:term1}, \eqref{eq:term2} and \eqref{eq:term3} and returning to \eqref{eq:big_ineq}, we now have that:
\begin{equation}\nonumber
        \bigg|\phi(x)|y_t(x')|^2 y_t(x)-\phi(x')|y_t(x)|^2 y_t(x')\bigg| \leq 
\end{equation}
\begin{equation}\nonumber
        \frac{L}{c}|y_t(x')|^2 \cdot \bigg( |y_t(x)| \cdot\bigg(|y_t(x)|+|y_t(x')|\bigg) + 
        \frac{C}{c}\bigg(2|y_t(x)|^2 + |y_t(x)|\cdot|y_t(x')|\bigg)\bigg)\cdot |x-x'| \leq
\end{equation}
\begin{equation}\nonumber
        \frac{L}{c}\bigg( 2 + 
        \frac{3C}{c}\bigg)|y_t(x')|^2 \cdot |y_t(x)|^2\cdot |x-x'|
\end{equation}
where the second inequality follows from noting that $|y_t(x)|\geq |y_t(x')|$. Dividing through by $|y_t(x')|^2|y_t(x)|^2$, we now see that for $x,x'\neq 0$:
\begin{equation}\nonumber
\begin{split}
    |v_t(x)-v_t(x')| = & \bigg|\phi(x')\frac{y_t(x')}{|y_t(x')|^2}-\phi(x)\frac{y_t(x)}{|y_t(x)|^2}\bigg|\\
    & \leq \frac{L}{c}\bigg( 2 + 
        \frac{3C}{c}\bigg) |x-x'|
\end{split}
\end{equation}
Thus, inside $B_R$, $v_t$ satisfies a Lipschitz condition with constant $a_1=\frac{L}{c}( 2 + \frac{3C}{c})$, as desired.\\
\\
As $v_t$ is Lipschitz continuous, we can now apply Picard–Lindel\"{o}f theorem to see that the below ordinary differential equation uniquely defines a continuous flow, $\Gamma_t$ over $t\in [0,1]$: 
\begin{equation}\nonumber
    \frac{\partial }{\partial t}\Gamma_t(x) = v_t(\Gamma_t(x)), \quad \Gamma_0(x)=x
\end{equation}
Now, let $r = Re^{-a_1}$ and assume $x_1,x_2 \in B_r$. As $B_r \subseteq B_R$, $v_t(x)$ satisfies the Lipschitz condition with constant $a_1$ inside $B_r$ and thus:
\begin{align*}
|\Gamma_t(x_1) - \Gamma_t(x_2)| &= \left|x_1 - x_2 + \int_0^t v_s(\Gamma_s(x_1))-v_s(\Gamma_s(x_2))\ ds\right| \\
&\leq |x_1 - x_2| + \int_0^t a_1|\Gamma_s(x_1)- \Gamma_s(x_2)|\ ds
\end{align*}
Applying Gr\"{o}nwall's inequality, we now obtain the following Lipschitz bound on $\Gamma_t$: 
\begin{equation}\nonumber
    |\Gamma_t(x_1) - \Gamma_t(x_2)| \le |x_1-x_2|e^{a_1t} \le |x_1-x_2|e^{a_1}
\end{equation}
And, thus:
\begin{align*}
|\Gamma_t(x_1) - \Gamma_t(x_2)|
&\geq |x_1 - x_2| - \int_0^t|v_s(\Gamma_s(x_1))-v_s(\Gamma_s(x_2))|\ ds\\
&\geq \bigg(1 - \int_0^t a_1e^{a_1s} ds\bigg)|x_1 - x_2| \\
& \geq (2-e^{a_1}) |x_1 - x_2|
\end{align*}
where the last inequality follows by noting $t\in[0,1]$ and $a_1>0$. It now follows that $\Gamma_t$ is bi-Lipschitz for $a_1$ small enough, and thus a homeomorphism. Further, we now have that for all $x \in B_r$:
\begin{equation}\nonumber
    |\Gamma_t(x)| \leq |x|e^{a_1 t} \leq Re^{-a_1}e^{a_1t} \leq R 
\end{equation}
Therefore $\Gamma_t(x)\in B_R$ for all $t\in [0,1]$ and $x \in B_r$. \\
\\
Now for $x \in B_r$, consider the function:
\begin{equation}\nonumber
    g_t(x) = \frac{1}{2}\Gamma_t(x)' H \Gamma_t(x) + t \phi(\Gamma_t(x)).
\end{equation}
Note that $g_t$ is exactly the composition $S_t\circ \Gamma_t$ discussed in Section \ref{sec:morse} in the main text (c.f. Fig. \ref{fig:morse_homotopy}). Differentiating $g_t$ with respect to $t$ we see that:
\begin{equation}\nonumber
    \frac{\partial}{\partial t} g_t(x) = \Gamma_t(x)'H \frac{\partial}{\partial t}\Gamma_t(x) + \phi(\Gamma_t(x)) + t\nabla \phi(\Gamma_t(x))\cdot\frac{\partial}{\partial t}\Gamma_t(x)
\end{equation}
\begin{equation}\nonumber
     = \Gamma_t(x)'H v_t(\Gamma_t(x)) + \phi(\Gamma_t(x)) + t\nabla \phi(\Gamma_t(x))' v_t(\Gamma_t(x))
\end{equation}
\begin{equation}\nonumber
     = y_t(\Gamma_t(x))' v_t(\Gamma_t(x)) + \phi(\Gamma_t(x)) 
\end{equation}
where the second equality follows by the definition of $y_t$. Substituting the definition of $v_t$ now yields:
\begin{equation}\nonumber
     = -\phi(\Gamma_t(x))\frac{y_t(\Gamma_t(x))'y_t(\Gamma_t(x))}{|y_t(\Gamma_t(x))|^2} + \phi(\Gamma_t(x)) = 0 
\end{equation}
Thus, for $x \in B_r$, $g_t(x)$ is constant over $t\in[0,1]$. Substituting $t=0$ and $t=1$, and denoting $\Gamma:=\Gamma_1$, we now obtain:
\begin{equation}\nonumber
    \frac{1}{2}x'Hx = \frac{1}{2}\Gamma_0(x)'H\Gamma_0(x) =g_0(x) = g_1(x) = \frac{1}{2}\Gamma(x)'H\Gamma(x) + \phi(\Gamma(x)) = f(\Gamma(x)),
\end{equation}
as desired. Finally, we show that $\Gamma$ is $C^1$. To do so it suffices to show that $\Gamma_t$ is $C^1$ with respect to $x$. By the `Smoothness of Flows' theorem of Section 17.6 of \cite{hirsch2004differential}, we have that $\Gamma_t$ is $C^1$ with respect to $x$ if $v_t$ is.\\
\\
To show continuous differentiability of $v_t$, we first note that $y_t$ is $C^1$ with respect to $x$, and thus so is $v_t$ for $x \neq 0$. To show $v_t$ is differentiable at $0$, first note that by Equation \ref{eq:v_cts} and the fact that $L$ can be made as small as desired by reducing the value of $R_0$ (see Lemma \ref{lem:lipschitz}), we have that:
\begin{equation}\nonumber
    \lim_{x\rightarrow 0} \frac{|\phi(x)|}{|y_t(x)|^2}=0
\end{equation}
Thus, for any $\vec{w}\in\mathbb{R}^{D}\neq 0$:
\begin{equation}\nonumber
    \nabla_{\vec{w}} v_t(0) = \lim_{h\rightarrow 0}\frac{v_t(h\vec{w})-v_t(0)}{h} = \bigg(\lim_{h\rightarrow 0 } \frac{\phi(h\vec{w})}{|y_t(h\vec{w})|^2}\bigg)\cdot \bigg(\lim_{h\rightarrow 0 } \frac{y_t(h\vec{w})-y_t(0)}{h}\bigg) = 0 \cdot \nabla_{\vec{w}} y_t(0) = 0
\end{equation}
where the second inequality follows from the definitions of $v_t$ and $y_t$. As this holds for any $\vec{w}\neq 0$, it follows that $\nabla v_t(0)=0$. To show that $\nabla v_t$ is continuous at $0$, we note that by Lipschitzness of $v_t$, for any $y\in B_{R}$:
\begin{equation}\nonumber
    |\nabla v_t (y)| = \bigg|v_t(y)-v_t(0)+ \frac{o(|y|)}{|y|}\bigg| \leq a_1 |y|+ \frac{o(|y|)}{|y|},
\end{equation}
where the equality follows from the definition of Jacobian. Thus:
\begin{equation}\nonumber
    \lim_{y\rightarrow 0}|\nabla v_t (y)| \leq \lim_{y\rightarrow 0} \bigg(a_1 |y| + \frac{o(|y|)}{|y|}\bigg)=0,
\end{equation}
as desired. It now follows that $v_t$, and thus $\Gamma$, are $C^1$.
\end{proof}

A standard corollary of the Morse Lemma is the following:

\begin{mdframed}
    \begin{corr}[Morse Points are Isolated]\label{corr:morse_isolated}
        If $S$ is a $C^k$ manifold, with $k\geq 2$ and $f:S \rightarrow \mathbb{R}$ is Morse, then the Morse points of $f$ are isolated. Further, if $S$ is compact then $f$ must have finitely many Morse points.
    \end{corr}
    \vspace*{0.3cm}
\end{mdframed}

\noindent
As noted above, this proof of the Morse lemma is useful for our purposes as it provides an explicit expression for the neighborhood over which $\Gamma$ is a bi-Lipschitz differomorphism. This is formalized by the below corollary. For concision in the following proof, and the following proof only, we use $\land$ to represent the minimum operator.

\begin{mdframed}
    \begin{corr}[Constants of the Morse Lemma]\label{corr:morse_constants}
        Suppose $r$ is a positive constant satisfying:
        \begin{equation}\nonumber
            \sup_{x\in B_r}||H_f(x)-H|| < ||H^{-1}||^{-1}\bigg[(1-m) \land  \frac{m^2\ln(2)}{6||H||||H^{-1}||}\bigg] \quad \text{ and }\quad r < \frac{m}{4||H||||H^{-1}||}
        \end{equation}
        for some $m\in (0,1)$, then $r$ satisfies the Morse Lemma. That is, $\Gamma$ is a well-defined bi-Lipschitz diffeomorphism on $B_{r}$.
    \end{corr}
    \vspace*{0.2cm}
\end{mdframed}
\begin{remark}
    It is worth emphasizing that the conclusion of Corollary \ref{corr:morse_constants} holds \text{for any} $m\in (0,1)$. For instance, choosing $m=\sfrac{1}{2}$, we see that the conditions reduce to:
        \begin{equation}\nonumber
            \sup_{x\in B_r}||H_f(x)-H|| < \frac{1}{2||H^{-1}||} \land  \frac{\ln(2)}{24||H||||H^{-1}||^2} \quad \text{ and }\quad r < \frac{1}{8||H||||H^{-1}||}.
        \end{equation}
    It follows that the value of $r$ depends only on the constants $||H||$ and $||H^{-1}||$, and the behavior of $H_f$ near the origin.
\end{remark}
\begin{proof}
    By tracking the constants in the proof of Lemma \ref{lem:morse}, we see that $\Gamma_t$ is a well-defined bi-Lipschitz diffeomorphism if the following criteria are met:
    \begin{equation}\nonumber
       \textbf{(1)}~~ L \leq ||H^{-1}||^{-1}, \quad \textbf{(2)}~~ 2-e^{a_1} > 0, \quad \text{ and } \quad \textbf{(3)}~~ r \leq \frac{c}{C}e^{-a_1}.
    \end{equation}
    where $L$ is the Lipschitz constant of $\phi$ on $B_r$, and $a_1, c$ and $C$ are as defined previous.\\
    \\
    Let $m\in(0,1)$ be arbitrary and let $r$ satisfy the statement of the corollary. For notational ease, let $\epsilon=m||H^{-1}||^{-1}$. By Lemma \ref{lem:lipschitz} and the assumption of the corollary, $L\leq\sup_{x\in B_r}||H_f(x)-H||<||H^{-1}||^{-1}-\epsilon$ and thus, the first criterion holds. Further, we also have that:
    \begin{equation}\nonumber
        L < \frac{\epsilon^2\text{ln}(2)}{6||H||} < \frac{\epsilon^2\text{ln}(2)}{2||H^{-1}||^{-1} + 3||H|| + L} = \frac{\epsilon^2\text{ln}(2)}{2(||H^{-1}||^{-1}-L) + 3(||H|| + L)},
    \end{equation}
    as, by common matrix identities and the above $L<||H^{-1}||^{-1}\leq||H||$. Thus, as $\epsilon < ||H^{-1}||^{-1}-L=c$, we have that:
    \begin{equation}\nonumber
        L < \frac{\epsilon^2\text{ln}(2)}{2c + 3C} < \frac{c^2\text{ln}(2)}{2c + 3C}.
    \end{equation}
    Rearranging and noting the definition of $a_1$, we have that $a_1<\text{ln}(2)$ and thus $2-e^{a_1}>0$, i.e. the second criterion holds.  Finally, through similar logic to the above, we have that:
    \begin{equation}\nonumber
        r < \frac{m||H^{-1}||^{-1}}{4||H||} \leq \frac{||H^{-1}||^{-1}-L}{2(||H||+L)} = \frac{1}{2}\cdot\frac{c}{C} \leq \frac{c}{C} e^{a_1},
    \end{equation}
    where the final inequality follows from the previous. Thus condition (3) holds. The result now follows.
\end{proof}
\vspace*{2cm}
\section*{Supplementary Material}

Supplementary material may be found at the following address:\\
\\
\textcolor{blue}{\underline{\href{https://www.overleaf.com/read/gxbjfpkqnshw\#3eee08}{https://www.overleaf.com/read/gxbjfpkqnshw\#3eee08}}}

\end{document}